\RequirePackage{fix-cm}
\documentclass[smallextended]{svjour3}       % onecolumn (second format)
\usepackage{graphicx}
\usepackage[plainpages=false,colorlinks=true,
            linkcolor=black,urlcolor=black,citecolor=black]{hyperref}
\usepackage{geometry}
\usepackage[utf8]{inputenc}

\usepackage{ushort}

\usepackage{amsmath,amssymb, amsfonts,mathrsfs,cite}
\usepackage{color}
\usepackage{booktabs}
\usepackage{tikz}
\usepackage{scrextend}
\usepackage[font=scriptsize,labelfont=bf]{caption}
\usepackage{authblk}
\usepackage{cite}
\usepackage{multirow}
\usepackage{algpseudocode}
\usepackage{graphicx}
\usepackage{array, tabularx}
\usepackage{algorithm,algcompatible}
\usepackage{booktabs} % for much better looking tables
\usepackage{paralist} % very flexible & customisable lists (eg. enumerate/itemize, etc.)
\usepackage{verbatim} % adds environment for commenting out blocks of text & for better verbatim
\usepackage{subfig} % mak
\usepackage{xcolor,marginnote,enumitem}
\usepackage[numbers,sort&compress]{natbib}

\numberwithin{equation}{section}

\DeclareMathOperator{\dom}{dom}
\DeclareMathOperator{\dist}{dist}

\DeclareMathOperator*{\argmin}{arg\,min}

\DeclareMathOperator{\Div}{div}

\DeclareMathOperator{\SRBGS}{SRBGS}

\newcommand{\bds}{\boldsymbol}

\newcommand{\mF}{\mathcal{F}}

\algnewcommand\INPUT{\item[\textbf{Input:}]}%
\algnewcommand\OUTPUT{\item[\textbf{Output:}]}%

\DeclareMathAlphabet{\mathbfit}{OML}{cmm}{b}{it}

\definecolor{Red}{rgb}{0.8,0,0}
\definecolor{Blue}{rgb}{0,0,0.8}
\definecolor{Green}{rgb}{0,0.4,0.4}
\definecolor{darkblue} {rgb} {0.00, 0.0, 1.0}

\title{A Preconditioned Difference of Convex Algorithm for Truncated Quadratic Regularization with Application to Imaging%\thanks{Grants or other notes
}
\titlerunning{Preconditioned DCA for truncated quadratic regularization}        % if too long for running head

\author{Shengxiang Deng         \and
	Hongpeng Sun %etc.
}

\institute{
	Shengxiang Deng \at
	Institute for Mathematical Sciences,
	Renmin University of China, Beijing, China.\\
	\email{2018103581@ruc.edu.cn}           %  \\
	\and
	Hongpeng Sun \at
	Institute for Mathematical Sciences,
	Renmin University of China, Beijing, China.\\
	\email{hpsun@amss.ac.cn}    
}

\date{Received: date / Accepted: date}
\ifpdf
\fi
\begin{document}

\maketitle

\begin{abstract}
	We consider the minimization problem with the truncated quadratic regularization, which is a nonsmooth and nonconvex problem. We cooperated the classical preconditioned iterations for linear equations into the nonlinear difference of convex functions algorithms with extrapolation. Especially, our preconditioned framework can deal with the large linear system efficiently which is usually expensive for computations.  Global convergence is guaranteed and local linear convergence rate is given based on the analysis of the Kurdyka-\L ojasiewicz exponent of the minimization functional. The proposed algorithm with preconditioners turns out to be very efficient for image restoration and is also appealing for image segmentation. 
\end{abstract}


%\paragraph{AMS subject classifications.}
%65K10, 	49J52,
%% 65K Mathematical programming, optimization and variational techniques
%%%  65K10 Optimization and variational techniques
%49M15
%% 49    Calculus of variations and optimal control; optimization
%% 49K          Optimality conditions
%%% 49K35  Minimax problems

\section{Introduction}\label{sec:intro}
 In this paper, we consider the truncated quadratic regularization with gradient operator for image restoration and segmentation
 \begin{align}
 &\argmin_{\boldsymbol{x} \in X} F(\bds{x}) = f(\bds{x})+ P^I(\bds{x}), \quad {P^I(\bds{x}):= \sum_{i=1}^m \sum_{j=1}^n\frac{\mu}{2}\min(|(\nabla \bds{x})_{i,j}|^2, \frac{\lambda}{\mu })},\label{equation:iso} \tag{ITQ} \\
 &\argmin_{ \bds{x} \in X} F(\bds{x}) = f(\bds{x}) + P^A(\bds{x}), \quad {P^A(\bds{x}):=\sum_{i=1}^m \sum_{j=1}^n \frac{\mu}{2}\sum_{l=1}^2\min(|(\nabla_l \bds{x})_{i,j}|^2, \frac{\lambda}{\mu })}, \label{equation:ani} \tag{ATQ}
 \end{align}
 { where $\lambda$ and $\mu$ are positive constants, $X:=\mathbb{R}^{m\times n} $ is a finite dimensional discrete image space, $\nabla = [\nabla_1, \nabla_2]^T$ and $f(\bds{x}) := \|A\bds{x} -\bds{x}_0\|_{2}^2/2$ with $A:X \rightarrow Y_0=\mathbb{R}^{m_0\times n_0} $ being a linear and bounded operator and $\bds{x}_0$ being the noisy or degraded image. Here and subsequently, the $|\cdot|$ norm denotes the usual Euclid norm which is also the length of the corresponding vector. For example, $|(\nabla \bds{x})_{i,j}|=\sqrt{(\nabla_1 \bds{x})_{i,j}^2 +(\nabla_2 \bds{x})_{i,j}^2  }$  for the isotropic case in \eqref{equation:iso} and $|(\nabla_l \bds{x})_{i,j}|$ is the absolute value of $(\nabla_l \bds{x})_{i,j}$ with $l=1$ or $l=2$ in \eqref{equation:ani}. } $P^I$ or $P^A$ is the isotropic or anisotropic truncated quadratic regularizations (abbreviated as ITQ or ATQ). The truncated quadratic (also called as half-quadratic) regularization has various applications in signal, image processing and computer vision \cite{AIG, AA, BVZ, BZ, GY,SC}. 
It was originated from the maximal posterior estimates for the Markov random fields within the probabilistic  
setting mainly the Bayesian framework \cite{GG}. It also appeared as the weak membrane energy  and the corresponding graduated non-convexity algorithm  developed in \cite{BZ}. The nonsmooth and nonconvex truncated quadratic regularization without gradient operator was also found in robust statistic where it can kill the outliers completely \cite{HRRS, GW}; see Figure \ref{fig:comparison:abs:TQ} for the absolute value function and the truncated quadratic function. The discrete truncated quadratic regularization can also be seen as the discrete version of the continuous variational Mumford-Shah functional \cite{CH, MS1, MS2, GW}. We refer to \cite{WLW} for the general framework of truncated regularization which covered the truncated quadratic problem. Due to so many important applications in imaging and other fields, there are already a lot of studies on algorithmic developments for this problems \cite{NN, GY}.  Generally, there are two categories of algorithms. One is the stochastic approximation approach including the simulated annealing and the  other is the deterministic approach.  There are many kinds of deterministic optimization algorithms including the graph-cut algorithm \cite{BVZ} and the graduated non-convexity algorithm (GNC) \cite{BZ}; see \cite{MN,CLMS} for its recent development. Fast algorithms are also developed in \cite{AA,AIG, CBAB1, CBAB2} which benefit from the alternating minimization technique by introducing some auxiliary variables \cite{GY,MD}.  

Inspired by the recent developments of the difference of convex algorithms (DCA) \cite{HAD, HAD1, HAD2, YU} and the powerful  Kurdyka-\L ojasiewicz (KL) analysis for nonconvex optimizations  \cite{AB, ABRC,ABS,LP, WCP} together with the preconditioned techniques in convex splitting algorithms \cite{BS1,BS2, BS3}, we tackle this problem by the proposed preconditioned DCA algorithm with extrapolation. {DCA is now widely used for analyzing and computing noncovex models in image and signal processing. For example, a weighted difference of anisotropc and isotropic TV model is proposed in \cite{LZOX} for better reconstruction  and a more delicate  $l_1$-$\alpha l_2$ model is further developed in \cite{LCGN}.} For \eqref{equation:iso} or \eqref{equation:ani}, we will employ the following difference of convex functions (DC) throughout this paper, $P^l(\bds{x}) = P_1^l(\bds{x})- P_2^l(\bds{x})$ with $l=I$ or $l=A$ and 
%\begin{equation}\label{eq:dc:convex}
\begin{align}
&P_1^I(\bds{x}) = \sum_{i=1}^m \sum_{j=1}^n\frac{\mu}{2}(|(\nabla \bds{x})_{i,j}|^2 +  \frac{\lambda}{\mu }), \ P_2^I(\bds{x}) =\sum_{i=1}^m \sum_{j=1}^n\frac{\mu}{2} \max {(|(\nabla \bds{x})_{i,j}|^2, \frac{\lambda}{\mu })}, \label{eq:dc:convex} \\
&P_1^A(\bds{x})= \sum_{i=1}^m \sum_{j=1}^n\frac{\mu}{2}\sum_{l=1}^2(|(\nabla_l \bds{x})_{i,j}|^2 +  \frac{\lambda}{\mu }) , \ P_2^A(\bds{x}) =\sum_{i=1}^m \sum_{j=1}^n\frac{\mu}{2} \sum_{l=1}^2 \max {(|(\nabla_l \bds{x})_{i,j}|^2, \frac{\lambda}{\mu })}. \notag
\end{align}
%\end{equation}
Note that both $f(\bds{x})$, $P_1^I(\bds{x})$ and $P_2^I(\bds{x})$ (or $P_1^A(\bds{x})$ and $P_2^A(\bds{x})$) are convex functions. $P_1^I$ (or $P_1^A$) is continuous differentiable with locally Lipschitz gradient and $P_2^I$ (or $P_2^A$) is proper closed function. Our motivation mainly comes from the challenging problem for solving the linear subproblems appeared in DCA, which is the most expensive step for DCA in a lot of applications \cite{HAD}. For example, splitting decomposition algorithm with error control is employed in \cite{HAD}.  We proposed a preconditioned framework and cooperated the preconditioned iteration for linear systems into the total nonlinear DCA iterations. In this framework, only one or few preconditioned steps are needed for the linear subproblems without solving it inexactly or exactly. Especially, the global convergence and the local linear convergent rate of DCA can also be obtained. Usually, the computational amount of one time or few times preconditioned iterations is quite less. For example, the computation effort of one Jacobi or one symmetric Gauss-Seidel iteration for large scale linear system is nearly negligible compared to solving the linear sytem even with moderate accuracy, especially for large scale linear system.

Our contributions belong to the following parts. First, we propose a preconditioned DCA for the truncated quadratic regularization with gradient operator including both the isotropic and anisotropic cases. With the classical preconditioning technique, we can deal with the large linear system efficiently for the nonlinear DCA algorithm with any finite time preconditioned iterations. No error control is needed for solving large linear systems while the convergence can be guaranteed. {For example, in the proposed preconditioned framework, one can still obtain global convergence of the DCA  by employing 10 specially designed symmetric red-black Gauss-Seidel iterations for the linear subproblem during each DCA iteration}. Second, with detailed analysis of the Kurdyka-\L ojasiewicz exponent of the minimization functional, together with the global convergence of the iterative sequence, we also prove the local linear convergence rate of the proposed preconditioned DCA. {Third, our global convergence and local convergence rate analysis is based on the difference of convex structure \eqref{eq:dc:convex} where $P_1$ ($P_1^I$ or $P_1^A$) has locally Lipschitz gradient and $P_2$ ($P_2^I$ or $P_2^I$) is closed and convex. This is different from the case in \cite{WCP} where $P_1$ is closed and convex and $P_2$ has locally Lipschitz gradient}.  Fourth, we also explore the feature of the truncated quadratic regularization for image segmentation within the proposed preconditioned DCA framework, which was already studied by a lot of algorithms including the graduated non-convexity algorithm \cite{BZ}, the graph-cut based discrete optimization method \cite{BVZ}, and the primal-dual first-order method \cite{SC}. 
Besides the image segmentation, it is known that the truncated quadratic regularization can also be used for image denoising. However, there is no systematic comparisons with the total variation regularization. We give some comparisons between the truncated quadratic regularization  and the total variation for image denoising with detailed parameters.

The rest of the paper is organized as follows. In section \ref{sec:theory}, after some preparations and the calculation of the Kurdyka-\L ojasiewicz exponent, we give the global convergence and present the local linear convergence rate of the proposed preconditioned and extrapolated DCA. 
In section \ref{sec:num}, we give a systematic  numerical study on the image denoising and image segmentation. Finally, we give some discussions on section \ref{sec:conclude}.
%\begin{figure}[!htbp]
%	\centering
%	\subfloat[$l_1$ regularization]{
%		\begin{minipage}{0.5\linewidth}
%			\centering
%			\includegraphics[width=0.5\textwidth]{plot_truncated/abs.pdf}
%		\end{minipage}
%	} 
%	\subfloat[truncated quadratic regularization]{
%		\begin{minipage}{0.5\linewidth}
%			\centering
%			\includegraphics[width=0.5\textwidth]{plot_truncated/truncated.pdf}
%		\end{minipage}
%	}
%\end{figure}
\begin{figure}%[!htb]
	%\graphicspath{{fig//}}
	\begin{center}
		\subfloat[$|x|$]
		{\includegraphics[width=0.35\textwidth]{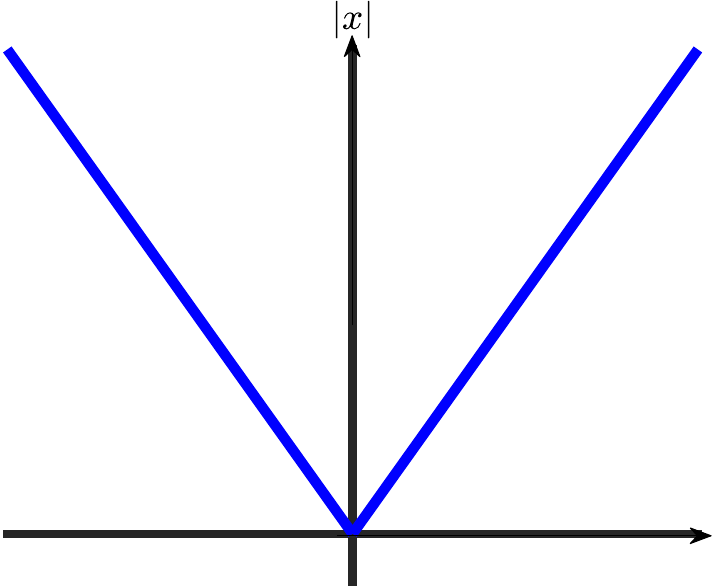}}\qquad
		\subfloat[ $\frac{1}{2}\min(|x|^2, \frac{\lambda}{\mu})$]
			{\includegraphics[width=0.35\textwidth]{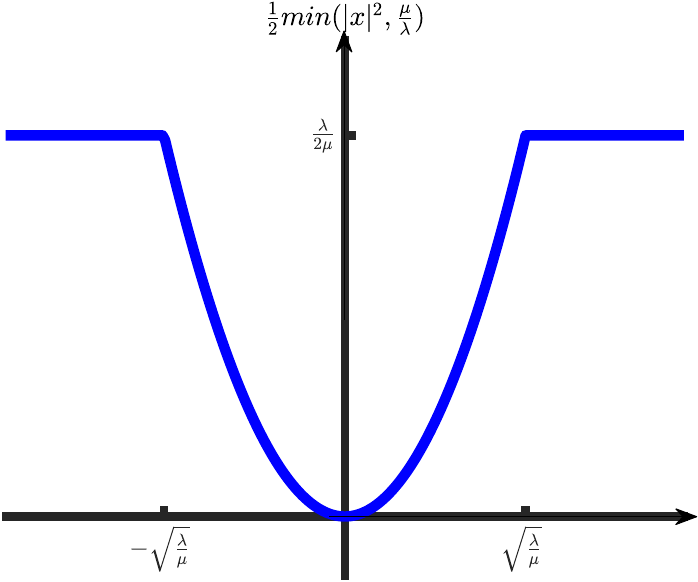}} 
	\end{center}
	\caption{\scriptsize{Absolute value function and truncated quadratic function in $\mathbb{R}$.} }
  \label{fig:comparison:abs:TQ}
\end{figure}
\section{Preconditioned DCA$_e$: convergence and preconditioners}\label{sec:theory}
\subsection{Preliminaries  and KL exponent analysis}

	Let $h: \mathbb{R}^n \rightarrow \mathbb{R}\cup \{+\infty\}$ be a proper lower semicontinuous function. Denote $\dom h: = \{ x \in \mathbb{R}^n : \ h(x) < +\infty  \}$.
  For each $x \in \dom h$, the limiting-subdifferential of $h$ at $x\in \mathbb{R}^n $, written $\partial h$, is defined as follows \cite{BM,Roc1}, 
      \[
      \partial h(x): = \left\{ \xi \in \mathbb{R}^n : \exists x_n \rightarrow x,   h(x_n) \rightarrow h(x),   \xi_n  \rightarrow \xi, \lim_{y\rightarrow x } \inf_{y \neq x_n}\frac{h(y)-h(x_n) - \langle \xi_n, y-x_n \rangle }{|y-x_n|}\geq 0 \right\}.
      \] 

It is known that the above subdifferential $\partial h$  reduces to the classical subdifferential in convex analysis when $h$ is convex. It can be seen that a necessary condition for $x \in \mathbb{R}^n$ to be a minimizer of $h$ is
$0 \in \partial h(x)$ \cite{AB}. 
For the global and local convergence analysis, we also need the Kurdyka-\L ojasiewicz (KL) property and KL exponent. 
\begin{definition}[KL property and KL exponent] \label{def:KL}
	A proper closed function $h$ is said to satisfy the KL property at $\bar x \in \dom \partial h$ if there exists $a \in (0, +\infty]$, a neighborhood $\mathcal{O}$ of $\bar x$, and a continuous concave function $\psi: [0, a) \rightarrow (0, +\infty)$ with $\psi(0)=0$ such that:
	\begin{itemize}
		 \item [{(i)}] $\psi$ is continuous differentiable on $(0,a)$ with $\psi'>0$.
		 \item [{(ii)}] For any $x \in \mathcal{O}$ with $h(\bar x) < h(x) <h(\bar x) + a$, one has
		 \begin{equation}\label{eq:kl:def}
		 \psi'(h(x)-h(\bar x)) \dist(0, \partial h(x)) \geq 1.
		 \end{equation}
	\end{itemize}
A proper closed function $h$ satisfying the KL property at all points in $\dom \partial h$ is called a KL function. If $\psi$ in \eqref{eq:kl:def} can be chosen as $\psi(s) = cs^{1-\theta}$ for some $\theta \in [0,1)$ and $c>0$, we say that $h$ satisfies KL properties at $\bar x$ with exponent $\theta$. This means that for some $\bar c >0$, we have
\begin{equation}\label{eq:KL:exponent:theta:exam}
\dist(0, \partial h(x)) \geq \bar c (h(x)-h(\bar x))^{\theta}.
\end{equation}
If $h$ satisfies KL property with exponent $\theta \in [0,1)$ at all the points of $\dom \partial h$, we call $h$ is a KL function with exponent $\theta$. 
\end{definition}
The following \emph{uniformized KL property} proved in \cite{BST} is also important for our discussions. 
\begin{lemma}\label{lem:uni:KL}
Assume $h$ is a proper closed function and $\Gamma $ is a compact set. If $h$ is a constant on $\Gamma$ and satisfies the KL property at each point of $\Gamma$, then there exist $\epsilon, a>0$ for any $\psi$ as in definition \ref{def:KL}, 
\begin{equation}\label{eq:uniform:KL}
		 \psi'(h(x)-h(\hat x)) \dist(0, \partial h(x)) \geq 1,
\end{equation} 
for any $\hat x\in \Gamma$ and any $x$ satisfying $\dist(x, \Gamma)< \epsilon$ and $h(\hat x)<h(x)<h(x)+a$.
\end{lemma}
%\subsection{global convergence}
  The minimization problem \eqref{equation:iso} or \eqref{equation:ani} is a standard DC programming and can be solved by DCA. {From now on, we will denote $x$ as the vectorized $\bds{x}$.  We will still use the same notations $A$, $\Delta$, $A^*$ and $\nabla$ (or $F$, $f$ and $P$) as the matrix version of the linear mappings (the functions) after vectorization}.  Let's  take the problem \eqref{equation:iso}  for example. The standard DCA iteration reads as follows,
  \begin{equation}\label{eq:iso:dca:exam}
  x^{t+1}: = \argmin_{x} f(x) + P_{1}^I(x) - \langle \xi^t, x \rangle, \quad  \xi^t \in \partial P_2^{I}(x)|_{x=x^t},
  \end{equation}
  where $P_1^I$, $f$ and $P_1^I$ are the same functions in \eqref{eq:dc:convex} and the term $\langle \xi^t, x \rangle$ essentially represents the linearization of the convex function $P_{1}^I(x)$ through its subgradient. It can be seen by replacing $\langle \xi^t, x \rangle$ by $\langle \xi^t, x -x^t \rangle + P_2^I(x^t)$ in \eqref{eq:iso:dca:exam} without changing the minimization problem \eqref{eq:iso:dca:exam}. By direct calculation, the minimizer $x^{k+1}$ of \eqref{eq:iso:dca:exam} can be obtained by solving the following linear equation during each DCA iteration
  \begin{equation}\label{eq:linear:DCA:ori}
  (A^*A-\mu \Delta)x = \xi^t + A^*x_0.
  \end{equation}
  It is very expensive and challenging to solve this kind of equation especially for large linear systems during each iteration even with error control. In \cite{HAD}, ``preconditioned decomposition algorithm" is employed to solve the equation with error control while $A$ is the identity operator in \eqref{eq:linear:DCA:ori}. Inspired by the preconditioned framework for the convex splitting algorithm \cite{BS1, BS2, BS3}, our motivation is to introduce  the powerful and classical preconditioning technique for linear systems such as \eqref{eq:linear:DCA:ori} and cooperate them into the nonlinear DCA.

   We introduce the preconditioned iterations for \eqref{eq:linear:DCA:ori} through  proximal terms with special metric (or weight).  Let's first introduce the inner product and norm induced by  the positive definite and self-adjoint operator (metric) $M$,
  \[
  \langle x, y \rangle_M: = \langle x,My \rangle, \quad \|x\|_M^2: = \langle x,Mx\rangle.   
  \]
Moreover, we will also employ the extrapolation framework that can bring out certain acceleration \cite{WCP} for a lot of cases. The extrapolation strategy is originated from  Nesterov's accelerated gradient method. To this end, let's introduce the extrapolation parameter $\beta$ such that  $\{\beta_{t}\}  \subseteq [0,1)$ and $\sup _{t} \beta_{t}<1$. The extrapolation step is done by $y^{t}= x^{t}+\beta_{t}(x^{t}-x^{t-1})$ where the previous iteration $x^{t-1}$ is incorporated.  With these preparations, we now give our algorithmic framework, i.e., the Algorithm \ref{alg:pre:dca}. Henceforth, we will consider the proposed Algorithm \ref{alg:pre:dca} with efficient preconditioners for solving the problem. 

%And we can use the proximal operate in $F(u)$, and rewrite
%\[
%F(u) = (P_1(u) + \frac{\|x\|_M^2}{2}) - (\frac{\|x\|_M^2}{2} -f(u) +P_2(u)).
%\] 

%\begin{equation} \label{eq:2}
%\begin{array}{l}{\text { Preconditioned difference-of-convex algorithm with extrapolation }\left(\mathrm{pDCA}_{e}\right):} 
%\\ {\text { Input: } x^{0} \in \operatorname{dom} P_{1},\left\{\beta_{t}\right\} \subseteq[0,1) \text { with } \sup _{t} \beta_{t}<1 . \text { Set } x^{-1}=x^{0}} 
%\\ {\text { for } t=0,1,2, \cdots} 
%\\ {\qquad 
%	\begin{aligned} 
%	y^{t}=& x^{t}+\beta_{t}\left(x^{t}-x^{t-1}\right) 
%	\\	x^{t+1} &=\underset{y \in \mathbb{R}^{n}}{\arg \min }\left\{\left\langle\nabla f\left(y^{t}\right)-\xi^{t}, y\right\rangle+\frac{1}{2}\left\|y-y^{t}\right\|_{M}^{2}+P_{1}(y)\right\}
%	\end{aligned}}
%\\ {\text{end for} }
%\end{array}
%\end{equation}
\begin{algorithm}
	\caption{Preconditioned difference-of-convex algorithm with extrapolation (preDCA$_{e}$) for $\argmin_x F(x) = f(x)+ P_1(x) - P_2(x)$\label{alg:pre:dca}}
	\begin{algorithmic}
		\STATE { $x^{0} \in \operatorname{dom} P_{1}$, $\{\beta_{t}\}  \subseteq [0,1)$, \text { with } $\sup _{t} \beta_{t}<1$. Set  $x^{-1}=x^{0}$. \\
			Iterate the following steps for $t=0, 1, \cdots$,
	%	\STATE {
		\begin{align}
		\xi^t &\in \partial P_2(x^t),  \\
	  y^{t}&= x^{t}+\beta_{t}\left(x^{t}-x^{t-1}\right) 
		\\	x^{t+1} &=\underset{y}{\arg \min }\left\{\left\langle\nabla f(y^{t})-\xi^{t}, y\right\rangle+\frac{1}{2}\|y-y^{t}\|_{M}^{2}+P_{1}(y)\right\}. \label{eq:proxi:y}
		\end{align}}
		%\STATE {$\qquad\left (\epsilon \Delta^{2}+(x^{2})^{\ast}x^{2}+\frac{1}{\epsilon_{0}} \big (Df_{0}(p)+ Dg_{0}(p)  \big ) +\frac{1}{\epsilon_{1}} \big ( D\hat{f}_{1}(p)+D\hat{g}_{1}(p\big) \right )d=b$}
%		\STATE{\textbf{Step 2: }Update $h^{l+1}$ by \eqref{eq:ufirst:h}} (or \eqref{eq:ufirst:h:ani} for the anisotropic case)
%		\STATE{\textbf{Step 3: }Project $h^{l+1}$ to the feasible set  $\{h:\norm[\infty]{h} \leq \alpha\}$, i.e., $h^{l+1}  =\mathcal{P}_{\alpha}(h^{l+1})$. Set $(u^{l+1},h^{l+1})$ as the initial value for the next Newton iteration and go to \textbf{Step 1}}.
		%		\IF{$\|F(p)\|_{D}\le \mathrm{tol}$}
		%		\STATE {$\qquad$ Decrease $\epsilon_{0}$ and $\epsilon_{1}$ using an appropriate strategy}
		%		\ENDIF      
		%	\ENDWHILE
		%\STATE{Recover denoised image $u$ as}
		\STATE{Unless some stopping criterion is satisfied, stop}
	\end{algorithmic}
\end{algorithm} 
Supposing the Lipschitz constant of $f$ in Algorithm \ref{alg:pre:dca} is $L$, {if choosing $M=L\bds{I}$ with $\bds{I}$ denoting the identity operator (or the identity matrix when vectoring $\bds{x}$)}, Algorithm \ref{alg:pre:dca} reduces to the proximal extrapolation DCA proposed in \cite{WCP} with different conditions on $P_1$ and $P_2$.  We employ the metric induced by $M$, which can bring out great flexibility to deal with the linear system with efficient preconditioners.   Let's take  the following Lemma  \ref{lem:feaible:percon}  for example to illustrate our motivation, {where we can reformulate \eqref{eq:proxi:y} as the classical preconditioned iteration \cite{SA}. }
\begin{lemma}\label{lem:feaible:percon}
	With appropriately chosen linear operator $M\geq L_0\bds{I}$ with positive constant $L_0\geq L$, the iteration \eqref{eq:proxi:y} actually can be reformulated as the following classical preconditioned iteration
	\begin{equation}\label{eq:pre:iter}
	x^{t+1}: = y^t + M_p^{-1}[b^t -Ty^t],
	\end{equation}
	where 
	\[
	b^t = L_0y^t -  \nabla f(y^{t})+\xi^{t}, \quad  T = L_0\bds{I} -\mu \Delta, \quad M_p = M -\mu\Delta \geq T.
    \]
\end{lemma}
\begin{proof}
	Denote $b_1^t  = \xi^t -\nabla f(y^t) $. By the structure of $P_1^A$ or $P_1^I$ in \eqref{eq:dc:convex}, we see
	\[
	M(y-y^t) -\mu \Delta y -b_1^t=0.
	\]
	We thus have
%	\begin{equation}
	\begin{align}
	x^{t+1} &= (M-\mu \Delta)^{-1}(b_1^t + My^t) \notag  \\  
	       & = (M-\mu \Delta)^{-1}( (M-\mu \Delta) y^t +b_1^t +\mu\Delta y^t ) \notag \\
	       & = y^t + (M-\mu \Delta)^{-1}[b_1^t +L_0y^t - (L_0 \bds{I}- \mu \Delta )y^t], \notag  \\ 
          &=y^t + (M-\mu \Delta)^{-1}[b^t  - (L_0\bds{I} - \mu \Delta )y^t],	       
	       	\end{align}
	  %     	\end{equation}
	which leads to \eqref{eq:pre:iter} with notation $M_p: = M  -\mu \Delta$. $M_p$ is actually a preconditioner for $T$ to solve the following linear equation
\begin{equation}\label{eq:orig:line:equa}
Tx = b^t.
\end{equation}
\qed
\end{proof}	
The following remark will give more interpretation of the preconditioned iteration \eqref{eq:pre:iter}.
\begin{remark}\label{rem:pre:gs}
Suppose the discretization of the operator $T=L_0\bds{I}-\mu \Delta$ in Lemma \ref{lem:feaible:percon} is $D-E-E^*$ (still denoting it as $T$ and using $\Delta$ as the discretized $\Delta$) where $D$ is the diagonal part, $-E$ represents the strict lower triangular part and $E^*$ is the transpose of $E$. If  choosing $M_p$ as the symmetric Gauss-Seidel preconditioner for $T$, it is well-known that \cite{SA} (chapter 4.1) (or \cite{BS1})
 \begin{align*}
 M_p = T+ E^*D^{-1}E.
 \end{align*}
 By Lemma \ref{lem:feaible:percon}, since $M_p= (M-\mu\Delta)$, we thus have the explicit form of $M$
 \[
 M =M_p + \mu \Delta = T+ E^*D^{-1}E + \mu \Delta  = T+ E^*D^{-1}E + L_0\bds{I} -(L_0\bds{I} - \mu \Delta)  = E^*D^{-1}E + L_0\bds{I}.
 \]
We also see $M \geq L_0 \bds{I}$ as in Lemma \ref{lem:feaible:percon}.  However, we do not need to calculate the explicit form of $M$ or $M_p^{-1}$, since the update \eqref{eq:pre:iter}  is exactly the one time symmetric Gauss-Seidel iteration for the linear equation $Ty = b^t$ \cite{SA}. {This means that  $x^{t+1}$ as in \eqref{eq:pre:iter} is also equivalent to \eqref{eq:proxi:y} through one time symmetric Gauss-Seidel iteration. } 
%Furthermore, it is proved that any finite number of the symmetric Gauss-Seidel preconditioned iteration is still equivalent to the update \eqref{eq:pre:iter} with  a generalized preconditioner $M_p$ \cite{BS1} and we still can have the corresponding metric $ M  \geq L_0 I$ (see Proposition 2.15 of \cite{BS1}). We conclude that any finite symmetric Gauss-Seidel preconditioned iteration can be reformulate as in \eqref{eq:pre:iter}  and can be cooperated into nonlinear DCA. }
\end{remark}

For image denosing problem, with $f(x) = {\|x-x_0\|_2^2}/{2}$ with Lipschitz constant $1$, if we choose $L_0 = \bds{I}$ in Lemma \ref{lem:feaible:percon}, the linear equation \eqref{eq:orig:line:equa} coincides with the original linear equation of DCA \eqref{eq:linear:DCA:ori}.  For image deblurring problem, one possible choice is that we can still use algorithm \ref{alg:pre:dca} with $f(x) = {\|Ax-x_0\|_2^2}/{2}$, where the symmetric Gauss-Seidel preconditoners can still be employed for the corresponding perturbed Laplacian  equation with using $\nabla f(y^t)$ explicitly in \eqref{eq:proxi:y}.  Here, we provide another choice. Taking the \eqref{equation:ani} for example, letting
	\begin{equation}\label{data:deblur}
	f=0, \quad  P_1(y) = {\|Ay-x_0\|_2^2}/{2} + P_1^A(y),  \quad P_2(y) = P_2^A(y),
	\end{equation} 
 we have the following proposition, whose proof is completely similar to Lemma \ref{lem:feaible:percon} and is thus omitted.
	\begin{proposition}\label{pro:feaible:percon:deblur}
			With appropriately chosen linear operator $M\geq L_0\bds{I}$ with positive constant $L_0\geq L$ and the data in \eqref{data:deblur}, the iteration \eqref{eq:proxi:y} in Algorithm \ref{alg:pre:dca} can be reformulated as the following classical preconditioned iteration
			\begin{equation}\label{eq:pre:iter:blur}
			x^{t+1}: = y^t + M_p^{-1}[b^t -Ty^t],
			\end{equation}
			where 
			\[
			b^t = L_0y^t  + A^*x_0 + \xi^{t}, \quad  T = L_0\bds{I} +A^*A -\mu \Delta, \quad M_p = M +A^*A -\mu\Delta \geq T.
			\]
    \end{proposition} 
	The condition $M\geq L_0\bds{I}$ comes from the positive definite requirement of $M$, which is important for the following convergence analysis. However, we can choose very small $L_0$ for the deblurring problem and $T$ can thus approximate the original linear system \eqref{eq:linear:DCA:ori}. Throughout this paper, if $f=0$ in \eqref{data:deblur} with Lipschitz constant $L=0$, we further assume $M\geq L_0 \bds{I}$ with constant $L_0>0$.

With Lemma \ref{lem:feaible:percon}, Remark \ref{rem:pre:gs}, and Proposition \ref{pro:feaible:percon:deblur}, it can be seen that one can cooperate the classical preconditioned iteration into the DCA framework through the proximal mapping with metric. We thus can deal with linear systems with powerful tools from the classical preconditioning techniques for linear algebraic equations. Now let's turn to the KL analysis for the convergence with our preconditioning framework. We begin with the KL exponent of the quadratic functions with an elementary proof.

\begin{lemma}\label{lemma:quadratic}
	The quadratic function ${q}(x)=\frac{1}{2}x^TQx-u^Tx+s$ is a KL function with KL exponent of $\frac{1}{2}$, where Q is a symmetric positive semidefinite matrix. Moreover, supposing that the minimal positive eigenvalue of $M$ is $\lambda_M$, then there exist small positive  $\varepsilon$ and $\eta$, such that for any $x$  satisfying $|x-\bar x|\leq \varepsilon$ and ${q}(\bar{x})< {q}(x)<{q}(\bar{x})+\eta$, we have
	\[
{q}(x)-{q}(\bar x)=	|{q}(x)-{q}(\bar x)| \leq \frac{1}{2\lambda_M} |\nabla {q}(x)|^2.
	\]
\end{lemma}
\begin{proof}
	First, noting that $\frac{1}{2}x^TQx-u^Tx+s$ and $\frac{1}{2}x^TQx-u^Tx$ have the same KL exponent, we just need to prove the case of the function ${q}(x)=\frac{1}{2}x^TQx-u^Tx$ without loss of generality. We first consider the case $\bar x$ such that $\nabla {q}(x)|_{x=\bar x}=0$, i.e., $Q\bar x=u$. Supposing $\lambda_1 \geq \lambda_2 \geq \cdots \geq \lambda_n \geq0$ are the eigenvalues of $Q$, we know $ \lambda_M= \min\{\lambda_i, \lambda_i  >0\}$ by assumption.  There exists an orthogonal matrix $P$ such that $Q = P^{-1}\text{Diag}[\lambda_1, \cdots, \lambda_n]P$. Furthermore, 
	\begin{equation*}	
	\begin{aligned}
	&|{q}(x)-{q}(\bar{x})|=|\frac{1}{2} \langle Q(x-\bar x), x-\bar x \rangle | \\
	&= \frac{1}{2}(x-\bar x)^TP^{-1}
	\begin{pmatrix}
	\lambda_1 \\
	& \ddots \\
	& & \lambda_n
	\end{pmatrix}
	P(x-\bar x) \le \frac{1}{2\lambda_M}(x-\bar x)^TP^{-1}
	\begin{pmatrix}
	\lambda^2_1\\
	& \ddots \\
	& & \lambda_n^2
	\end{pmatrix}
	P(x-\bar x)\\
	&= \frac{1}{2\lambda_M} \langle Q(x-\bar x),Q(x-\bar x)\rangle =  \frac{1}{2\lambda_M} \langle Qx-u,Qx-u\rangle = \frac{1}{2\lambda_M} | \nabla {q}(x)|^2.
	\end{aligned}
	\end{equation*}
Now, let's turn to the case $|\nabla {q}(\bar x)| =| Q\bar x- u| =\delta_0 > 0$. 
%	We see that we can find a constat $c,\epsilon,\eta>0$ which can satisfying the KL property for any $x$ satisfying $\|x-\bar{x}\|\le\epsilon$ and $f(\bar{x})< f(x)<f(\bar{x})+\eta$.
Supposing $|x-\bar x|\leq \varepsilon$, we see
	\begin{equation}\label{KL:fne0}
	\begin{aligned}
	|{q}(x)-{q}(\bar{x})|&=|\frac{1}{2}x^TQx-u^Tx-\frac{1}{2}\bar{x}^TQ\bar{x}+u^T\bar x|\\
	&=|\langle\frac{1}{2}(x-\bar{x})^TQ(x-\bar{x})+ \langle Q\bar{x}-u, x-\bar x\rangle| \\
	& \leq \frac{1}{2}\|Q\|\varepsilon^2 + \delta_0 \varepsilon.
	\end{aligned}
	\end{equation}
	For $|\nabla {q}(x)|^2$, we have
		\begin{equation}
	\begin{aligned}
	&|\nabla {q}(x)|^2=|Qx-u|^2 = |Qx - Q\bar x + Q \bar x -u|^2 \\
	&= |Qx - Q\bar x|^2 + 2 \langle Q(x -\bar x), Q \bar x -u \rangle +  |Q \bar x -u |^2 
	\geq \delta_0^2 -\|Q\|^2 \varepsilon^2 -  2  \delta_0 \|Q\| \varepsilon.
	\end{aligned}
	\end{equation}
	To obtain $|{q}(x)-{q}(\bar{x})| \le \frac{1}{2\lambda_M} | \nabla f(x)|^2$,  one can choose 
	\[
	(\delta_0^2 -\|Q\|^2 \varepsilon^2 -  2  \delta_0 \|Q\| \varepsilon) \frac{1}{2\lambda_M} \geq \frac{\delta_0^2}{2} \frac{1}{2\lambda_M} \geq
	\frac{1}{2}\|Q\|\varepsilon^2 + \delta_0 \varepsilon,
	\]
	which leads to 
	\[
	\varepsilon \leq \varDelta_0:=\min\left(\frac{\delta_0}{\|Q\|}, \frac{\delta_0}{\|Q\|} (\sqrt{\frac{\|Q\|}{2\lambda_M} + 1} -1)\right).
	\]
	We thus have $|{q}(x)-{q}(\bar{x})| \le \frac{1}{2\lambda_M} | \nabla {q}(x)|^2$ for all $|x-\bar x|\leq \varDelta_0$. The proof is complete. 
%	we can see that the limit of the right of (\ref{KL:fne0}) is 0 when $x \to \bar{x}$,but the limit of the left of (\ref{KL:fne0}) is not 0 when $x\to\bar{x}$ due to $\partial f(\bar{x})\ne 0$.
%	so we can find $c,\epsilon,\eta>0$ which can satisfying the KL property for any $x$ satisfying $\|x-\bar{x}\|\le\epsilon$ and $f(\bar{x})< f(x)<f(\bar{x})+\eta$.
\qed
\end{proof}
\begin{remark}\label{rem:notnew}
Lemma \ref{lemma:quadratic} can be seen as a special case of Corollary 5.1 of \cite{LP}, which originated from \cite{LL} for the convex quadratic problem. 
\end{remark}	

	We now discuss the KL exponent of the truncated quadratic regularization functional \eqref{equation:iso} and \eqref{equation:ani}. We will employ the recent study on KL analysis of the functions which can be written as minimization of a finite number of KL functions with KL exponent $1/2$; see \cite{LP}.  Let's turn to the following theorem.

\begin{theorem}\label{thm:KL}
		Assuming the linear operators $A:X\rightarrow Y_0=\mathbb{R}^{m_0\times n_0}$, $K_l: X \rightarrow Y_l=\mathbb{R}^{m_l\times n_l \times c_l}$, $l=1,\cdots, k$ are linear, bounded operators and $\mu_l$, $\tau_l$ are positive parameters, then the KL exponent of the following general truncated quadratic regularization functional $F(\bds{x})$ is ${1}/{2}$,
		\begin{equation}\label{eq:l:min:real}
		F(\bds{x}) = \frac{\|A\bds{x}-\bds{x}_0\|_2^2}{2 } +  \sum_{l=1}^k\sum_{i=1}^{m_l}\sum_{j=1}^{n_l           }  \frac{\mu_l}{2}\min(|(K_l\bds{x})_{i,j}|^2, \tau_l).
		\end{equation}
\end{theorem}

\begin{proof}
Let's first vectorize $\bds{x}$ and $\bds{x}_0$ as the column vector ${x}\in \mathbb{R}^{mn}$ and  ${x}_0 \in \mathbb{R}^{m_0n_0}$ correspondingly. We will still use $A$, $K_l$, $l=1,\cdots, k$ as the discrete matrix versions of the corresponding linear operators.  The equation \eqref{eq:l:min:real} then becomes
\begin{equation}\label{eq:l:min:real:vector}
F({x}) = \frac{\|A{x}-{x}_0\|_{2}^2}{2 } +  \sum_{l=1}^k\sum_{i=1}^{m_ln_l}  \frac{\mu_l}{2}\min(|(K_l{x})_{i}|^2, \tau_l),
\end{equation}
where $(K_l{x})_{i}$ denote the $i$-th component of  $K_l{x}$. Note the fact that 
\[
\min(a,b)+\min(c,d) = \min(a+c,a+d, b+c, b+d), \quad \forall a, b, c, d \in \mathbb{R}.
\]
Similarly, for the summation with $N:=\sum_{l=1}^{k}m_ln_l$ terms with each term of the form $\min(|(K_l{x})_{i}|^2, \tau_l)$ as in \eqref{eq:l:min:real:vector}, we can rewrite $F(x)$ as follows 
\begin{equation}\label{eq:l:min:real:vector:union}
F({x}) = \frac{\|A{x}-{x}_0\|_{2}^2}{2 } +  \min_{1\leq i \leq 2^N}P_{i}(x).
\end{equation}
$P_{i}(x)$ comes from summing the  selected term  $|(K_l{x})_{i}|^2$ or $\tau_l$ from $\min(|(K_l{x})_{i}|^2, \tau_l)$ for $l=1, \cdots, k$ and $i=1, \cdots, m_ln_l$. For example, we can choose 
\[
P_1(x) = \sum_{l=1}^k\sum_{i=1}^{m_ln_l} \tau_l, \quad P_{2}(x) = \sum_{l=1}^k\sum_{i=1}^{m_ln_l} |(K_l{x})_{i}|^2.
\]
All the other $P_i(x)$ with $i=3, \cdots, 2^N$ can be chosen similarly. Furthermore, it can be readily checked that each $P_{i}(x)$, $i=1, \cdots, 2^N$, is a convex quadratic function. It is straightforward that  \eqref{eq:l:min:real:vector:union} can be written as
\begin{equation}\label{eq:mini:F}
F({x}) = \min_{1\leq i \leq 2^N}F_i(x), \quad F_i(x):= \frac{\|A{x}-{x}_0\|_{2}^2}{2 }+  P_{i}(x), \quad i=1, \cdots, 2^N.
\end{equation}
Actually, we can reformulate each $F_i(x)$ in the form of quadratic function as in Lemma \ref{lemma:quadratic}. Taking the function $F_2(x)$ for example, let
\begin{align*}
&\Lambda = [A/\sqrt{2}, K_{1,1}, \cdots, K_{1,m_1n_1}, \cdots,K_{l,1}\cdots, K_{l,m_ln_l},\cdots, K_{k,1}\cdots, K_{k,m_kn_k}]^T,  \\
&b=[x_0/\sqrt{2}, 0, \cdots, 0]^T \in \mathbb{R}^{N_0}, \quad N_0:=m_0n_0+\sum_{l=1}^km_ln_lc_l,
\end{align*}
where $K_{i_1,i_2}x=(K_{i_1}x)_{i_2}$, $i_1=1, \cdots, k$, and $i_2=1, \cdots, m_{i_1}n_{i_1}$.
We can thus rewrite $F_2(x)$ as follows
\[
P_2(x) =\|\Lambda x -b\|_{2}^2,
\]
which is clearly a quadratic function. Since each $F_i(x)$ is a quadratic function as in Lemma \ref{lemma:quadratic}, then each $F_i(x)$ has KL exponent $1/2$ by Lemma \ref{lemma:quadratic}. 
With \cite{LP} (Theorem 3.1) and noting $F(x)$ is a continuous function, we conclude that   $F({x})$ is a KL function with an exponent $1/2$, since it can be written as minimization of $F_i(x)$ with KL exponent of $1/2$ in \eqref{eq:mini:F}.
\qed
\end{proof}	

\begin{remark}
For the isotropic model \eqref{equation:iso}, we can choose $m_0=m$, $n_0=n$, $K_1=[\nabla_1, \nabla_2
]$ with $m_1=n$, $n_1=n$, $c_1=2$ and $k=1$ as in \eqref{eq:l:min:real}. For the anisotropic model \eqref{equation:ani}, we can choose  $m_0=m$, $n_0=n$, $K_1=\nabla_1$ and $K_2=\nabla_2$ with $m_1=m_2=m$, $n_1=n_2=n$, $c_1=c_2=1$ and $k=2$ as in \eqref{eq:l:min:real}.
\end{remark}

Henceforth, we will make extensive use of the following auxiliary function
\begin{equation}\label{auxiliaryfunction}
	{E(x,y) = f(x) + P(x) + \frac{1}{2}\|x-y\|_M^2 = F(x)+ \frac{1}{2}\|x-y\|_M^2.} 
\end{equation} 
Let's calculate the exponent of KL inequality of the auxiliary function $E(x,y)$ in  \eqref{auxiliaryfunction} at the stationary point.
We do this through the relationship between the original function $F(x)$ and the auxiliary function $E(x,y)$.
\begin{lemma}\label{lemma:relationship}
	If a proper closed function ${F}(x)$ has the KL property at a stationary point $\bar{x}$ with an exponent of $\frac{1}{2}$,  then the auxiliary function $E(x,y)={F}(x) + \frac{1}{2}\|x-y\|_M^2$ has the KL property at the stationary point {$(\bar x,\bar x)$} with the exponent of $\frac{1}{2}$.
\end{lemma}
\begin{proof}
	Because $\bar{x}$ is a stationary point of $F(x)$, we  have $0 \in \partial F(\bar x)$. Supposing $0 \in \partial E(\bar{x},\bar{y}) = ({\partial F(\bar{x})+M(\bar{x}-\bar{y})}, {M(\bar{y}-\bar{x})})^T$, we have $\bar{x} = \bar{y}$ by $M \geq L_0\bds{I}$. Since $F$ has the KL property at $\bar{x}$ with the exponent $\frac{1}{2}$, there exist $c_1$, $\epsilon$ and $\eta >0$ such that 
	\begin{equation} \label{KL:f(x)}
	(F(x)-F(\bar{x}))\le c_1 \dist^2(0,\partial F(x)),
	\end{equation}
	whenever $x\in \dom \partial F(x)$, $\|x- \bar{x}\| \le \epsilon$ and $F(\bar{x})<F(x)<F(\bar{x})+\eta$.
	We thus have 
	\begin{equation}
	\begin{aligned}\label{dist:left}
	|E(x,y) -E(\bar{x},\bar{x})| 
	&\le |F(x)-F(\bar{x})|+\frac{1}{2}\|x-y\|_M^2 \\
	&\le c_1 \dist^2(0,\partial F(x))+\frac{1}{2}\|x-y\|_M^2
	\end{aligned}
	\end{equation}
	for any $(x,y)$ satisfying $x\in \dom \partial F$, $\|x - \bar{x}\|\le \epsilon$, $\|y-\bar{x}\|\le\epsilon$ and $E(\bar{x},\bar{x})<E(x,y)<E(\bar{x},\bar{x})+\eta$.
	Furthermore, if there exists a positive constant $c_2$ such that 
	\begin{equation}
	\begin{aligned}\label{KL:E(x,y)}
	c_1 \dist^2(0,\partial F(x))+\frac{1}{2}\|x-y\|_M^2&\le c_2 \dist^2(0,\partial E(x,y))
	\\ &= c_2 \dist^2((0,0)^T,(\partial F(x)+M(x-y),M(y-x))^T) ,
	\end{aligned}
	\end{equation}
	we get the lemma. For any $\epsilon >0$, we have
	\begin{equation}\label{dist:right}
	\begin{aligned}
	\dist^2(0,\partial E(x,y))&= 2\|M(y-x)\|^2+\inf_{\xi\in\partial F(x)}(\|\xi\|^2+\langle\xi,x-y\rangle_M)\\
	&\ge 2\|M(y-x)\|^2+\inf_{\xi\in\partial F(x)}\left[\|\xi\|^2-(\alpha\|\xi\|^2+\frac{1}{\alpha}\|M(x-y)\|^2)\right]\\
	&=(2-\frac{1}{\alpha})\|M(y-x)\|^2+(1-\alpha)\dist^2(0,\partial F(x)) \\
	& \geq (2-\frac{1}{\alpha}) \ushort \lambda_M\|y-x\|_M^2+(1-\alpha)\dist^2(0,\partial F(x))
	\end{aligned}
	\end{equation}
	where the first inequality follows from the inequality $ab\ge -(\alpha a^2 + \frac{1}{\alpha}b^2)$, $ \forall \alpha>0$ and $\ushort \lambda_M$ is the minimum positive eigenvalue of $M$ as before. Setting $\frac{1}{2}<\alpha<1$,  we have $1-\alpha>0$ and $2-\frac{1}{\alpha}>0$. With \eqref{dist:left} and \eqref{dist:right}, to obtain \eqref{KL:E(x,y)}, one can fix $c_2$ as follows
	\begin{equation}
	\frac{1}{2}\le c_2(2-\frac{1}{\alpha})\ushort \lambda_M, \quad
	c_1\le c_2(1-\alpha) \Rightarrow c_2 \geq \max(\frac{c_1}{1-\alpha}, \frac{\alpha}{(4\alpha-2)  \ushort \lambda_M}) \geq 0.\label{c2:c1}
	\end{equation}
	We thus get 
	\begin{equation}\label{eq:c2}
	|E(x,y)-E(\bar{x},\bar{x})|\le c_2\dist^2(0,\partial E(x,y)),
	\end{equation}
	and the lemma follows.
	\qed
\end{proof}

\subsection{Global convergence and local convergence rate}
Recall that $\bar{x}$ is a stationary point of $F$ if $0\in \partial F(\bar{x})$.
We will first study a property of the iteration \eqref{eq:proxi:y}. We further assume $F$ is level-bounded (see Definition 1.8 \cite{Roc1}), i.e., lev$_{\leq \alpha}F: =\{ x: F( x) \leq \alpha\}$ is bounded (or possibly empty). We employ the similar idea in  \cite{WCP} with different conditions on $P_1$ and $P_2$ here.
\begin{proposition}\label{pro:1}
	The right hand-side of \eqref{eq:proxi:y}: $g(x):=\left\langle\nabla f(y^{t})-\xi^{t}, x \right\rangle+\frac{1}{2}\|x-y^{t}\|_M^{2}+P_{1}(x)$ is a  strongly convex function. Moreover, $g(x^{t+1})\le g(x^t) -\frac{1}{2}\|x^{t+1}-x^t\|_M$ when $x^{t+1}$ is a stationary point of $g(x)$.
\end{proposition} 
\begin{proof} For any $\xi_1 \in \partial P_1(x)$, by the convexity of $\frac{1}{2}\|x-y^{t}\|_M$ and $P_1(x)$ on $x$, we have
	\begin{align}
	g(y)-g(x)&=\left\langle\nabla f(y^{t})-\xi^{t}, y-x \right\rangle+\frac{1}{2}\|y-y^{t}\|_M^{2}-\frac{1}{2}\|x-y^{t}\|_M^{2}+P_{1}(y)-P_1(x) \notag \\
	&\ge \left\langle\nabla f(y^{t})-\xi^{t}, y-x \right\rangle+\frac{1}{2}\|y-x\|_M^2+\left\langle x-y^t,y-x\right\rangle_M+\left\langle \xi_1,y-x\right\rangle \notag \\
	&=\left\langle\nabla f(y^{t})-\xi^{t}+M(x-y^t)+ \xi_1,y-x\right\rangle+\frac{1}{2}\|y-x\|_M^2 \label{eq:descent:strong} \\
	&\geq \left\langle\nabla f(y^{t})-\xi^{t}+M(x-y^t)+ \xi_1,y-x\right\rangle+\frac{L}{2}\|y-x\|^2, \ \ \forall x,y\in \dom g. \notag
	\end{align}
	Since 
	\[
	\nabla f(y^{t})-\xi^{t}+M(x-y^t)+ \xi_1 \in \partial g(x),
	\]
	we see $g(x)$ is a strongly convex function with a modulus that is not less than $L_0$. Setting $x=x^{t+1}$ and $y=x^t$, by the fact that $0\in \partial g(x)|x = x^{t+1}$, according to \eqref{eq:proxi:y}, together with \eqref{eq:descent:strong}, we have
	\begin{equation}\label{eq:descent:down}
	g(x^{t+1})\le g(x^t) -\frac{1}{2}\|x^{t+1}-x^t\|_M.
	\end{equation}   
\qed
\end{proof}
We will first show that the sequence $\{ {x^t}\}$ generated by the proposed algorithm \ref{alg:pre:dca}  converges to a stationary point of $E(x,y)$.

\begin{theorem}\label{thm:1}
	Let ${x^t}$ be a sequence generated by $preDCA_e$ for solving the minimization problem \eqref{equation:iso} or \eqref{equation:ani}. Then the following statements hold:
	\begin{itemize}
		\item [\emph{(i)}] $\displaystyle{\lim_{t \to \infty}\|x^{t+1}-x^t\|_M = 0} $,
		\item [\emph{(ii)}]  The limit $\displaystyle{\lim_{k\to\infty}E(x^t,x^{t-1}) =:\zeta}$
		\text{exists and} $E\equiv \zeta$ {on} $\Upsilon$. Henceforth, we denote  $\Upsilon$ as the set of accumulation points of the sequence $(x^t,x^{t-1})$. 
	\end{itemize}
\end{theorem}
\begin{proof}
	We first prove (i). By Proposition \ref{pro:1}, we can get 
	\begin{multline} \label{ineq:1}
	\left\langle\nabla f(y^t)-\xi^t,x^t\right\rangle + \frac{1}{2}\|x^{t+1} - y^t\|_M^2 +P_1(x^{t+1}) \\
	\le \left\langle\nabla f(y^t)-\xi^t,x^{t+1}\right\rangle + \frac{1}{2}\|x^t - y^t\|_M^2 +P_1(x^t) - \frac{1}{2}\|x^{t+1} - x^t\|_M^2 .
	\end{multline}
	On the other hand, since $\nabla f$ is Lipschitz continuous with a modulus of $L$,  we have
	%	\begin{equation} 
	\begin{align}
	&f(x^{t+1}) + P(x^{t+1}) \le f(y^t) + \langle\nabla f(y^t),x^{t+1} -y^t\rangle +\frac{L}{2}\|x^{t+1} -y^t\|^2 +P_1(x^{t+1}) - P_2(x^{t+1}) \notag\\
	&\le f(y^t) + \langle\nabla f(y^t),x^{t+1} -y^t\rangle +\frac{1}{2}\|x^{t+1} -y^t\|_M^2 +P_1(x^{t+1}) - P_2(x^{t+1}) \notag \\
	&\le f(y^t) + \langle\nabla f(y^t),x^{t+1} -y^t\rangle +\frac{1}{2}\|x^{t+1} -y^t\|_M^2 +P_1(x^{t+1}) - P_2(x^t)-\langle\xi^t,x^{t+1}-x^t\rangle \notag \\
	&\le f(y^t) + \langle\nabla f(y^t),x^t-y^t\rangle+\frac{1}{2}\|x^t-y^t\|_M^2 +P_1(x^t)-P_2(x^t)-\frac{1}{2}\|x^{t+1}-x^t\|_M^2 \notag \\
	&\le f(x^t) +P(x^t)+\frac{1}{2}\|x^t-y^t\|_M^2 -\frac{1}{2}\|x^{t+1}-x^t\|_M^2, \label{ineq:2}
	\end{align}
	%\end{equation} 
	where the second inequality follows from $M\geq L_0 \bds{I} \geq L \bds{I}$, the third one comes from the fact that $\xi^t \in \partial P_2(x^t)$, the fourth inequality follows from \eqref{eq:descent:down} and the fifth one by the convexity of $f$. From \eqref{ineq:2}, we have
	\begin{equation*}
	f(x^{t+1})+P(x^{t+1})\le f(x^t)+P(x^t)+\frac{1}{2}\beta_{t}^2\|x^t-x^{t-1}\|_M^2-\frac{1}{2}\|x^{t+1}-x^t\|_M^2.
	\end{equation*}
	Then, we can obtain that
	\begin{align}
	&\frac{1}{2}(1-\beta_{t}^2)\|x^t-x^{t-1}\|_M^2\le \left[f(x^t)+P(x^t)+\frac{1}{2}\|x^t-x^{t-1}\|_M^2\right] \notag\\
	&-\left[f(x^{t+1}) +P(x^{t+1})+\frac{1}{2}\|x^{t+1}-x^t\|_M^2\right]=E(x^t,x^{t-1})-E(x^{t+1},x^t). \label{ineq:3}
	\end{align}
	Since  ${\beta_{t}}\in\left[0,1\right)$, we see from (\ref{ineq:3}) that ${f(x^t)+P(x^t)+\frac{1}{2}\|x^t-x^{t-1}\|_M^2}$ is nonincreasing. We can thus get that
	\[
	f(x^t)+P(x^t)\le f(x^t)+P(x^t)+\frac{1}{2}\|x^t-x^{t-1}\|_M^2\le f(x^0)+P(x^0), \ \ \forall t\ge 0,
	\]
	which shows that ${x^t}$ is bounded by the level-boundedness of $F$ (Definition 1.8 of \cite{Roc1} and \cite{WCP}) and $F(x)\ge 0$. Then summing up both sides of (\ref{ineq:3}) from $t=0$ to $\infty$, we obtain
	\begin{align*}
	\frac{1}{2}\sum_{t=0}^{\infty}(1-\beta_{t}^2)\|x^t-x^{t-1}\|_M^2 &\le f(x^0)+P(x^0)-\liminf_{t \to \infty }\left[f(x^{t+1})+P(x^{t+1})+\frac{1}{2}\|x^{t+1}-x^t\|_M^2\right] \\
	&\le f(x^0)+P(x^0) < \infty .
	\end{align*}
	Since $\sup_{t} \beta_{t} <1$, we deduce from the above inequation that $ \sum_{t=1}^{\infty}\|x^t-x^{t-1}\|_M^2<\infty$ and $\lim_{t \to \infty}\|x^{t+1}-x^t\|_M^2=0$. This proves (i).
	
	Now we prove (ii), it can be seen that the sequence ${E(x^t,x^{t-1})}$ is nonincreasing form (\ref{ineq:3}). Together with the fact that $\Upsilon$ is a nonempty compact set due to ${x^t}$ is bounded, we conclude that $\zeta :=\lim_{k\to\infty}E(x^t,x^{t-1})$ exists.
	Now, let's show $E\equiv \zeta$ on $\Upsilon$. Taking any $(\bar{x},\bar{x})\in\Upsilon$, there exists a convergent subsequence $(x^{t_i},x^{t_i-1}) $ such that $\lim_{i \to \infty}(x^{t_i},x^{t_i-1}) =(\bar{x},\bar{x})$. Using the fact that $x^{t_i}$ is the minimizer of the subproblem in \eqref{eq:proxi:y}, we have
	\begin{align*}
	&P_{1}\left(x^{t_{i}}\right)+\left\langle\nabla f\left(y^{t_{i}-1}\right)-\xi^{t_{i}-1}, x^{t_{i}}\right\rangle+\frac{1}{2} \|x^{t_{i}}-y^{t_{i}-1}\|_M^{2} \\
	&\leq P_{1}(\bar{x})+\left\langle\nabla f\left(y^{t_{i}-1}\right)-\xi^{t_{i}-1}, \bar{x}\right\rangle+\frac{1}{2} \|\bar{x}-y^{t_{i}-1}\|_M^{2}.
	\end{align*}
	Rearranging terms above, we obtain  
	\begin{equation}\label{eq:expan:p1}
	P_{1}\left(x^{t_{i}}\right)+\left\langle\nabla f\left(y^{t_{i}-1}\right)-\xi^{t_{i}-1}, x^{t_{i}}-\bar{x}\right\rangle+\frac{1}{2}\|x^{t_{i}}-y^{t_{i}-1}\|_M^{2} \leq P_{1}(\bar{x})+\frac{1}{2} \|\bar{x}-y^{t_{i}-1}\|_M^{2}.
	\end{equation}
	Furthermore, we observe 
	\begin{equation*}
	\begin{aligned}
	\|\bar{x}-y^{t_i-1}\|_M&=\|\bar{x}-x^{t_i}+x^{t_i}-y^{t_i-1}\|_M\le\|\bar{x}-x^{t_i}\|_M+\|x^{t_i}-y^{t_i-1}\|_M \\
	&=\|\bar{x}-x^{t_i}\|_M+\left\|x^{t_{i}}-x^{t_{i}-1}-\beta_{t_{i}-1}\left(x^{t_{i}-1}-x^{t_{i}-2}\right)\right\|_M \\
	&\le \|\bar{x}-x^{t_i}\|_M+\|x^{t_i}-x^{t_i-1}\|_M+\|x^{t_i-1}-x^{t_i-2}\|_M.
	\end{aligned}
	\end{equation*}
	Since $\|x^{t+1}-x^t\|_M\to 0$ and $\lim_{i\to\infty}x^{t_i}=\bar{x}$, we have
	\begin{equation*}
	\|\bar{x}-y^{t_{i}-1} \|_M \to 0 \  \text { and } \ \|x^{t_{i}}-y^{t_{i}-1} \|_M \to 0.
	\end{equation*}
	Moreover, with \eqref{eq:expan:p1}, we obtain
	\begin{equation*}
	\begin{aligned} 
	\zeta &=\lim _{i \to \infty} f\left(x^{t_{i}}\right)+P\left(x^{t_{i}}\right) \\ 
	&=\lim _{i \to \infty} f\left(x^{t_{i}}\right)+P\left(x^{t_{i}}\right)+\left\langle\nabla f\left(y^{t_{i}-1}\right)-\xi^{t_{i}-1}, x^{t_{i}}-\bar{x}\right\rangle+\frac{1}{2}\left\|x^{t_{i}}-y^{t_{i}-1}\right\|_M^{2} \\ 
	& \leq \limsup _{i \to \infty} f\left(x^{t_{i}}\right)+P_{1}(\bar{x})-P_{2}\left(x^{t_{i}}\right)+\frac{1}{2}\left\|\bar{x}-y^{t_{i}-1}\right\|_M^{2}=F(\bar{x}).
	\end{aligned}
	\end{equation*}
	Since $F$ is lower semicontinuous, we have
	\begin{equation}
	F(\bar{x}) \leq \liminf _{i \to \infty} F\left(x^{t_{i}}\right)=\lim _{i \to \infty} F\left(x^{t_{i}}\right)=\zeta.
	\end{equation}
	Consequently, $F(\bar{x})=\liminf_{i\to\infty}F(x^{t_i})=\zeta$. Noting that for any $(\bar{x},\bar{x})\in\Upsilon$, we have $E(\bar{x},\bar{x})=F(\bar{x})=\zeta$. We thus conclude $E\equiv\zeta$ on $\Upsilon$ and (ii) follows.
	\qed
\end{proof}

	\begin{theorem}\label{thm:stationary}
		Any accumulation point of   ${x^t}$ is a stationary point of $F$. Furthermore, we have $\sum^\infty_{k=1}\|x^t-x^{t-1}\|$$\leq\infty$.
	\end{theorem}

\begin{proof}
	With the same assumption of Theorem \ref{thm:1}, let $\bar{x}$ be an accumulation of ${x^t}$. By the first-order optimality condition of the subproblem \eqref{eq:proxi:y}, we get 
	\begin{equation*}
	-M(x^{t+1}-y^t) \in \nabla P_1(x^{t+1}) + \nabla f(y^t) -\xi^t. 
	\end{equation*}
	With the fact $ y^{t} = x^t + \beta_{t}(x^t-x^{t-2})$, we obtain that\\
	\begin{equation} \label{ineq:converge condition}
	-M[(x^{t+1}-x^t)-\beta_t(x^t-x^{t-1})] \in \nabla P_1(x^{t+1}) +\nabla f(y^t) -\xi^t.
	\end{equation}
	Because of the convexity of $P_2$ and the the boundeness of ${x^t}$, by passing to a subsequence if necessary, then $\lim_{i \to \infty}\xi^t$ exists without loss of generality, which belongs to $ \partial P_2(\bar{x})$ due to the closedness of $ \partial P_2$ (Theorem 8.6 \cite{Roc1}). Using the fact that $\|x^{t+1}-x^t\|_M^2 \to 0$ from Theorem \ref{thm:1} (ii) together with the closedness of $\nabla P_1$ and $\nabla f$, we get upon passing to the limit in \eqref{ineq:converge condition} that 
	\[
	0 \in \nabla P_1(\bar{x}) + \nabla f(\bar{x}) -\partial P_2(\bar{x}).
	\] 
	Then, considering the subdifferential of the function $E(x,y)$ at the point $(x^t,x^{t-1})$, we have 
	\begin{equation}
	\partial E\left(x^{t}, x^{t-1}\right)=\left(\nabla f(x^{t})+\nabla P_{1}(x^{t})+M(x^{t}-x^{t-1})-\partial P_{2}(x^{t}),  -M(x^{t}-x^{t-1})\right)^T.
	\end{equation}
	On the other hand, with \eqref{ineq:converge condition} and the fact $\xi^t \in \partial P_2(x^t)$,
%	 we have 
%	\[
%	M(x^{t+1}-x^t)+\nabla f(y^t)+\nabla P_1(x^{t+1})\in\partial P_2(x^t).
%	\]
%	Under this relation,‘
	 we have 
	\begin{align*}
	&(M(x^t-x^{t+1} + (1+ \beta_t)(x^{t}-x^{t-1}))+\nabla f(x^t)-\nabla f(y^t)+\nabla P_1(x^t)-\nabla P_1(x^{t+1}), \\
	&-M(x^t-x^{t-1}))^T \in\partial E(x^t,x^{t-1}).
	\end{align*}
	Together with the fact that $\nabla f, \nabla P_1$ is Lipschitz continuous on a bounded set and $M \geq L \bds{I}$, we see that there exists $C_0>0$ such that
	\begin{equation}\label{distance:partialE}
	\begin{aligned}
	\dist((0,0),\partial E(x^t,x^{t-1}))&\le C_0(\|x^t-x^{t-1}\|_M+\|x^{t+1}-x^{t}\|_M)\\
	&\leq C(\|x^t-x^{t-1}\|+\|x^{t+1}-x^{t}\|),
	\end{aligned}
	\end{equation}
	where the constant $C$ depending on $M$ and $C_0$.
	We rewrite \eqref{ineq:3} as
	\begin{equation}\label{E:le}
	E(x^{t}, x^{t-1})-E(x^{t+1}, x^{t}) \geq D_0 \|x^{t}-x^{t-1}\|_M^{2} \geq D \|x^{t}-x^{t-1}\|^{2}.
	\end{equation}
	Then, we first consider the case that there exists a $t>0$ such that $E(x^t,x^{t-1})=\zeta$. Since ${E(x^t,x^{t-1})}$ is decreasing with the limit $\zeta$,  we thus have $E(\bar{x}^t,\bar{x}^{t-1})=\zeta$ for any $\bar{t}>t$.
	 Hence, $\sum_{t=0}^\infty \|x^t-x^{t-1}\|_M <\infty$ follows easily.
	We next consider the case that $E(x^t,x^{t-1})>\zeta$, $\forall t>0$. Since $E$ is a KL function and $E\equiv \zeta$ on $\Upsilon$, by Lemma \ref{lem:uni:KL}, there exist an $\epsilon>0$ and a continuous concave function $\psi$ with $a>0$ such that
	\begin{equation}\label{E:KL}
	\psi'(E(x,y)-\zeta)\text{dist}((0,0),\partial E(x,y)) \ge 1,\ \ \forall (x,y)\in U,
	\end{equation}
	where $U=\left\{(x, y) \in \mathbb{R}^{n} \times \mathbb{R}^{n}: \operatorname{dist}((x, y), \Upsilon)<\epsilon\right\}\cap$ $\left\{(x, y) \in \mathbb{R}^{n} \times \mathbb{R}^{n}: \zeta<E(x, y)<\zeta+a\right\}$.
	Moreover, we can get that there exists $T>0$ such that
	\begin{equation}\label{KL:phi}
	\psi^{\prime}\left(E(x^{t}, x^{t-1})-\zeta\right) \cdot \operatorname{dist}\left((0,0), \partial E(x^{t}, x^{t-1})\right) \geq 1, \ \ \forall t\ge T.
	\end{equation}
	Due to $\lim_{t \to \infty}\text{dist}((x^t,x^{t-1}),\Upsilon) =0$, there thus exists $T_1>0$ such that $\dist((x^t,x^{t-1}),\Upsilon)<\epsilon$ \text{whenvere} $t\ge T_1$.
	From the concavity of $\psi$, we see that 
	\begin{equation*}
	\begin{array}{l}{\left[\psi\left(E(x^{t}, x^{t-1})-\zeta\right)-\psi\left(E(x^{t+1}, x^{t})-\zeta\right)\right] \cdot \operatorname{dist}\left((0,0), \partial E(x^{t}, x^{t-1})\right)} \\ {\left.\geq \psi^{\prime}\left(E(x^{t}, x^{t-1})-\zeta\right)\right) \cdot \operatorname{dist}\left((0,0), \partial E(x^{t}, x^{t-1})\right) \cdot\left[E(x^{t}, x^{t-1})-E(x^{t+1}, x^{t})\right]} \\ {\geq E(x^{t}, x^{t-1})-E(x^{t+1}, x^{t})}.\end{array}
	\end{equation*}
	Combining this with (\ref{distance:partialE}) and (\ref{E:le}), we can get that for any $t\ge T$,
	\begin{align*}
	\|x^{t}-x^{t-1}\|^{2} \leq & \frac{C}{D}\left[\psi\left(E(x^{t}, x^{t-1})-\zeta\right)-\psi\left(E(x^{t+1}, x^{t})-\zeta\right)\right] \\
	& \cdot\left(\|x^{t}-x^{t-1}\|+\|x^{t-1}-x^{t-2}\|\right).
	\end{align*}
	Moreover, we can see further that (by the inequality $a \leq \sqrt{cd}$ $\Rightarrow$  $a \leq c + \frac{d}{4}$ for $a, b,c \geq 0$)
	\begin{equation}\label{ineq:xt}
	\begin{aligned}
	\frac{1}{2}\|x^{t}-x^{t-1}\| \leq& \frac{C}{D}\left[\psi\left(E(x^{t}, x^{t-1})-\zeta\right)-\psi\left(E(x^{t+1}, x^{t})-\zeta\right)\right] \\
	&+\frac{1}{4}\left(\|x^{t-1}-x^{t-2}\| - \|x^{t}-x^{t-1}\|\right).
	\end{aligned}
	\end{equation}
	Summing up the above relation from $t=T$ to $\infty$, we have
	\begin{equation}
	\sum_{t=T}^{\infty}\|x^{t}-x^{t-1}\| \leq \frac{2 C}{D} \psi\left(E(x^{T}, x^{T-1})-\zeta\right)+\frac{1}{2}\|x^{T-1}-x^{\bar{T}-2}\|<\infty.
	\end{equation}
 Thus $\{ x^t\}$ is a Cauchy sequence and its global convergence follows.
 \qed 
\end{proof}	

\begin{remark}\label{rem:wcp}
{Actually, the proofs in Theorems \ref{thm:1} and \ref{thm:stationary} have a lot of differences from the proofs in \cite{WCP}. These are mainly because of two reasons. The first is the proximal term $\|y-y^t\|_{M}^2/2$ designed for preconditioning which is different from \cite{WCP} where $M=L\bds{I}$. The second is the conditions on the functions $P_1$ and $P_2$ are different from \cite{WCP}  as mentioned in section \ref{sec:intro}.}
\end{remark}

We next consider the convergence rate of the sequence $\{x^t\}$ under the condition that the auxiliary function $E$ is a KL function at the stationary point whose $\psi $ takes the form $\psi(s)=cs^{1-\theta}$ for $\theta=\frac{1}{2}$, {which can be guaranteed by Theorem \ref{thm:KL} and Lemma \ref{lemma:relationship}}. This kind of convergence rate analysis is standard; see \cite{AB,ABRC,LP,WCP} for more comprehensive analysis. We follow a similar line of arguments for the local convergence analysis based on the KL property.

\begin{theorem}\label{thm:converge rate}
	Let ${x^t}$ be a sequence generated by preDCA$_{e}$ for solving \eqref{equation:iso} or \eqref{equation:ani} and suppose that ${x^t}$ converges to some $\bar{x}$.  Since $E$ is a KL function with $\psi$ in KL inequality taking the form $\psi(s)=cs^{1-\theta}$ for $\theta=\frac{1}{2}$ and $c>0$ at the stationary point, then there exist $c_1>0,t_0>0$ and $ \eta\in(0,1)$ such that $\|x^t -\bar{x}\| <c_1\eta^t $ for $\forall t>t_0$.
\end{theorem}

\begin{proof}
	If there exists $t_0>0$ such that $E(x^{t_0},x^{t_0-1})=\zeta$, then one can show that ${x^t}$ is finitely convergent as before and the local linear convergence holds trivially. Hence, we only consider the case when $E(x^{t},x^{t-1}) > \zeta$, $\forall t>0$. 
	Define $\Delta_t=E(x^t,x^{t-1})-\zeta$ \text{and} $S_t=\sum_{i=t}^\infty\|x^{i+1}-x^i\|$, where $S_t$ is well-define thanks to Theorem \ref{thm:1} (ii). Then, using (\ref{ineq:xt}), we have for any $t>T$ that
	\begin{align*}
	S_t&=2\sum_{i=t}^\infty\frac{1}{2} \|x^{i+1}-x^i\| \le2\sum_{i=t}^\infty\frac{1}{2}\|x^i-x^{i-1}\|\\
	&\le 2\sum_{i=t}^{\infty}\left[\frac{C}{D}\left[\phi(E(x^{i}, x^{i-1})-\zeta)-\phi(E(x^{i+1}, x^{i})-\zeta)\right]+\frac{1}{4}(\|x^{i-1}-x^{i-2}\|-\|x^{i}-x^{i-1}\|)\right] \\
	&\le \frac{2C}{D}\phi(E(x^t,x^{t-1})-\zeta)+\frac{1}{2} \|x^{t-1}-x^{t-2}\|\\
	&=\frac{2C}{D}\phi(\Delta_t)+\frac{1}{2}(S_{t-2}-S_{t-1})\le\frac{2C}{D}\phi(\Delta_t)+\frac{1}{2}(S_{t-2}-S_t),
	\end{align*}
	where the last inequality follows from the fact that ${S_t}$ is nonincreasing.  By \eqref{KL:phi} with $\psi(s)=cs^{\frac{1}{2}}$,  for all sufficiently large $t$,
	\[
	\frac{c}{2}\Delta_t^{-\frac{1}{2}}\text{dist}((0,0),\partial E(x^t,x^{t-1}))\ge 1.
	\]
	 Rewriting (\ref{distance:partialE}) by the definition of $S_t$, we see that for all sufficiently large $t$,
	\[
	\text{dist}((0,0),\partial E(x^t,x^{t-1}))\le C(S_{t-2}-S_t).
	\]
	We thus can get 
	\[
	(\Delta_t)^{\frac{1}{2}}\le \frac{Cc}{2}(S_{t-2}-S_t).
	\]
	Combining this with $S_t\le \frac{2C}{D}\phi(\Delta_t)+\frac{1}{2}(S_{t-2}-S_t)$, we see that for all sufficiently large $t$,
	\begin{equation}
	S_t\le C_1(S_{t-2}-S_t)+\frac{1}{2}(S_{t-2}-S_t)=(C_1+\frac{1}{2})(S_{t-2}-S_t),
	\end{equation}
	where $C_1 =\frac{c^2C^2}{D}$. Hence,
	\begin{equation}\label{eq:rate:upper:bound}
	\|x^t-\bar{x}\| \le\sum_{i=t}^\infty \|x^{i+1}-x^i\| =S_t\le S_{t_1-2}\eta
	^{t-t_1+1},\quad \eta : =\sqrt{\frac{2C_1+1}{2C_1+3}},
	\end{equation}
	which completes the proof.
	\qed
\end{proof}

\begin{remark}
	As $L_0$ in Lemma \ref{lem:feaible:percon} is sufficiently large, the upper bound of the convergence rate $\eta$ in \eqref{eq:rate:upper:bound}  would decrease as the condition number of $M$ increases.
\end{remark}

\begin{proof}
 Suppose the minmial and maximal eigenvalues of $M$ are $ \ushort{\lambda}_M$ and $\bar \lambda_M$.	We can see that the convergence rate  is related to $c,C$ and $D$ from \eqref{eq:rate:upper:bound}. Firstly, we see that $c$ is not related to $M$ for large $L_0$, since $\frac{c_1}{1-\alpha} \ge \frac{\alpha}{(4\alpha-2) \ushort\lambda_M}$ by \eqref{c2:c1}, \eqref{eq:c2} and $M\geq L_0\bds{I}$ when $L_0$ is large enough. Note that here $c$ is related to $c_2$ in \eqref{eq:c2}. Furthermore, we can choose $ D= \ushort\lambda_MD_0$ from (\ref{E:le}) and $C =\sqrt{\bar{\lambda}_M}C_0$ from \eqref{distance:partialE} and the fact $ \bar{\lambda}_M\|x\|^2 \ge \|x\|_M^2 \geq \ushort \lambda_M \|x\|^2$. Since $C_0,D_0$ is not related to $M$, we see $C_1=\frac{c^2C^2}{D}=\frac{c^2C_0^2}{ D_0} \frac{\bar \lambda_M}{\ushort{\lambda}_M}$ would increase when the condition number of $M$ increases. Thus the upper bound of the convergence rate $ \sqrt{1-\frac{2}{2C_1+3}}$ is decreased when the condition number of $M$ increases.
 \qed
\end{proof}

\subsection{Preconditioners and Preconditioned DCA$_e$}
Let's first consider the convex subdifferentials $\partial P_2^I$ or $\partial P_2^A$ by the following lemma for more general case.
\begin{lemma}
	The subdifferential of the convex function $p(x):=\max(|Kx|^2, \tau)/2$ is as follows
	\begin{equation}\label{eq:sub:deri:max:1}
	\left\{K^* \chi^s_{K,\tau} K x\ | \ s\in[0,1]\right\} = \partial_{x}(\frac{1}{2}\max(|Kx|^2, \tau)),
 	\end{equation}
	where the constant $\tau >0$ and $\chi^s_{K, \tau}$ is the generalized Clarke derivatives of $\max(\cdot, 1.0)$,
	\begin{equation}\label{eq:newton:deri:max}
	\chi^s_{K,\tau} = \begin{cases}
	1, \quad & |K x | /\sqrt{\tau} >1.0, \\
	s, \quad & |K x | /\sqrt{\tau}=1.0, \  s\in [0,1], \\ 
	0, \quad & |K x | / \sqrt{\tau} <1.0.
	\end{cases}
	\end{equation}
	Furthermore, we have
	\begin{equation}\label{eq:multi:sub:max}
	\partial \left(  \sum_{i=1}^l \frac{\mu_i}{2}\max(|K_ix|^2, \tau_i) \right)=\left\{ \sum_{i=1}^l\mu_i K_i^* \chi^{s^{i}}_{K_i,\tau_i}K_ix: \ \  s^i\in[0,1], \ \ i=1, \cdots, l\right\}.
	\end{equation} 
	Henceforth, we choose $s^i\equiv1$, $i=1, \cdots, l$ throughout this paper.
\end{lemma}
\begin{proof}
	We mainly need to consider \eqref{eq:newton:deri:max}. Since for each $p_i(x):=\frac{\mu_i}{2}\max(|K_i x|^2, \tau_i)$, $i=1, \cdots, l$, $\dom p_i = X$ which is the whole domain, then by \cite{Roc} (Theorem 23.8),  we have
	\[
	\partial (\sum_{i=1}^l p_i(x)) = \sum_{i=1}^l \partial p_i(x).
	\] 
	Let's consider the Clarke's generalized subdifferential of $p(x)$. Denote $p_1(x) = \frac{1}{2}|Kx|^2$ and $p_2(x)=0$. It can be seen  that $p(x)$ is a $PC^1$ function \cite{Sch}. It can be easily checked that while $|Kx|> \sqrt{\tau}$, 
	\begin{align*}
	\langle \nabla_x (\frac{1}{2}|Kx|^2), y\rangle = \langle Kx, Ky \rangle,
	\end{align*}
	where the inner product above is understood in the usual vector inner product such as $a^Tb$.
	We thus have $(\nabla_x p_1)(y) = 	\chi^s_{K,\tau}\langle Kx, K y \rangle $ with $s=1$ for $|Kx|> \sqrt{\tau}$.  $\nabla_x p_2(x)=0$ follows easily. We thus conclude that \cite{Sch} (Proposition 4.3.1)
	\[
	\partial_x p(x) = \text{co}\{\nabla_x p_1(x), \nabla_x p_2(x)\},
	\]
	{where the notation  ``$\text{co}$" denotes}  the convex hull of the corresponding set \cite{CL}.
	Since for convex functions, the Clarke generalized subdifferential concides with their convex subdifferential \cite{CL} (Proposition 2.2.7), we have \eqref{eq:sub:deri:max:1}. 
	\qed
\end{proof}
Now we turn to the preconditioners for image denoising. According to Lemma \ref{lem:feaible:percon}, we call a preconditioner $M_p$ {\it feasible} for $T$ if and only if 
\[
M_p \geq T= L_0\bds{I} - \mu \Delta,
\]
where $L_0$ is the same as in Lemma \ref{lem:feaible:percon}. 
For operators of type $T = \alpha \bds{I} - \beta \Delta$ for $\alpha,\beta > 0$
where $\Delta = \Div \nabla$ can be interpreted as a discrete
Laplace operator with homogeneous Neumann boundary
conditions \cite{BS1, BS2}. In other words: solving $Tx = b$ correspond to
a discrete version of the boundary value problem
\begin{equation}\label{neumann:boudary}
\begin{cases}
\alpha \bds{x} - \beta \Delta \bds{x} = b,\\
\frac{\partial \bds{x} }{\partial \nu}|_{\partial \Omega} =  0.
\end{cases}
\end{equation}
Besides Remark \ref{rem:pre:gs}, here are some examples from the classical iterative methods for linear systems.
%\subsection{Preconditioners}
\begin{example}\mbox{}
	\label{ex:standard-precon}
	\begin{itemize}
		\item
		Obviously, $M_p = T$ with $L_0=L$ is a feasible preconditioner for $T$ in \eqref{eq:pre:iter:pre}.
		This choice reproduces
		the original proximal DCA with $M=L\bds{I}$ without preconditioners. 
		\item
		The choice $M_p = c \bds{I}$ with
		$c \geq L + \mu \|\nabla\|^2$ also yields a feasible
		preconditioner. This is corresponding to the Richardson method, where the update for $x^{k+1}$ can be seen as an
		explicit step.
	\end{itemize}
\end{example}
We employ the efficient symmetric Red-Black Gauss-Seidel (SRBGS) iterations as the preconditioner \cite{BS1,BS2}.  Of course, several steps of this preconditioner can also be
performed; see the following Proposition \ref{multiple:precon}. Furthermore, we denote
the $n$-fold application of the symmetric Red-Black to the
initial guess $x$ and right-hand side $b$ by \cite{BS1,BS2}
\begin{equation}
\label{eq:symmetric-red-black-gauss-seidel}
\SRBGS_{\alpha,\beta}^n(x, b) = (\bds{I} + M_p^{-1}(\mathbf{1}_b - T))^nx
\end{equation}
making it again explicit that $M_p$ and $T$ depend on $\alpha$ and $\beta$.
\begin{proposition}[\cite{BS1}]\label{multiple:precon}
	Let $M_p$ be a feasible preconditioner for $T$ and $n \geq 1$. Then,
	applying the preconditioner $n$ times, i.e.,
	\[
	\left\{
	\begin{aligned}
	% b^k &= \bar x^k - \sigma K^*\bar y^k \\
	x^{k+(i+1)/n} &= x^{k+i/n} + M_p^{-1}(b^k - Tx^{k+i/n}) \\
	i &= 0,\ldots,n-1
	\end{aligned}\right.
	\]
	corresponds to $x^{k+1} = x^k + M_{p,n}^{-1}(b^k - Tx^k)$ where
	$M_{p,n}$ is a feasible preconditioner.
\end{proposition}
It is proved in \cite{BS1} that $ M_{p,n} \geq T$. We thus conclude that the corresponding metric in the proximal term in \eqref{eq:proxi:y} $M_{n}$ is positive definite, since $ M_{n} = M_{p,n}+ \mu \Delta  \geq T + \mu \Delta \geq L_0\bds{I} $. {Proposition \ref{multiple:precon} provides great flexibility for choosing how many inner preconditioned iterations for the linear subproblems. }

Remembering $\nabla x = [\nabla_1 x, \nabla_2 x]^T$ and $|\nabla x|^2 = |\nabla_1 x|^2 + |\nabla_2 x|^2$, let's denote
\[
\chi_{x}=\begin{cases}
1, \quad |\nabla x| \geq \sqrt{\frac{\lambda}{\mu}}, \\
0, \quad |\nabla x| < \sqrt{\frac{\lambda}{\mu}},
\end{cases}
\
\chi_{x,1}=\begin{cases}
1, \quad |\nabla_1 x| \geq \sqrt{\frac{\lambda}{\mu}}, \\
0, \quad |\nabla_1 x| < \sqrt{\frac{\lambda}{\mu}},
\end{cases}
\
\chi_{x,2}=\begin{cases}
1, \quad |\nabla_2 x| \geq \sqrt{\frac{\lambda}{\mu}}, \\
0, \quad |\nabla_2 x| < \sqrt{\frac{\lambda}{\mu}}.
\end{cases}
\] 
With these preparations, we give the following Algorithm \ref{alg:dca:deno:scale}.
\begin{algorithm}
	\caption{preDCA$_{e}$ for image denoising or segmentation of the truncated model \eqref{equation:iso} or \eqref{equation:ani} with $A=\bds{I}$\label{alg:dca:deno:scale}}
	\begin{algorithmic}
		\STATE { $x^{0} \in \operatorname{dom} P_{1}$, $\{\beta_{t}\}  \subseteq [0,1)$, \text { with } $\sup _{t} \beta_{t}<1$. Choose $L_0\geq L$ and set  $x^{-1}=x^{0}$. \\
			Iterate the following steps for $t=0, 1, \cdots$, 
			%				\STATE {
			%					\begin{equation}
			%	
			%					\end{equation}	}
			\begin{align}
			\xi^t& = \begin{cases}
			\nabla^*\chi_{x^t}\nabla x^t, \quad \text{for the isotropic case,} \\ 
		(\nabla_1^*\chi_{x^t,1}\nabla_1 + \nabla_2^*\chi_{x^t,2}\nabla_2) x^t, \quad \text{for the anisotropic case,} 
			\end{cases}\\
			y^{t}&= x^{t}+\beta_{t}(x^{t}-x^{t-1}), \notag
			\\
			b^t &= (L_0-I)y^t +\nabla f|_{y=y^t} + \xi^t,  \\
			x^{t+1} &= \SRBGS_{\alpha,\beta}^n(y^t, b^t), \quad T: = (L_0\bds{I} -\mu\Delta).\label{eq:pre:iter:pre} %\label{eq:proxi:y}
			\end{align}}
		%\STATE {$\qquad\left (\epsilon \Delta^{2}+(x^{2})^{\ast}x^{2}+\frac{1}{\epsilon_{0}} \big (Df_{0}(p)+ Dg_{0}(p)  \big ) +\frac{1}{\epsilon_{1}} \big ( D\hat{f}_{1}(p)+D\hat{g}_{1}(p\big) \right )d=b$}
		%		\STATE{\textbf{Step 2: }Update $h^{l+1}$ by \eqref{eq:ufirst:h}} (or \eqref{eq:ufirst:h:ani} for the anisotropic case)
		%		\STATE{\textbf{Step 3: }Project $h^{l+1}$ to the feasible set  $\{h:\norm[\infty]{h} \leq \alpha\}$, i.e., $h^{l+1}  =\mathcal{P}_{\alpha}(h^{l+1})$. Set $(u^{l+1},h^{l+1})$ as the initial value for the next Newton iteration and go to \textbf{Step 1}}.
		%		\IF{$\|F(p)\|_{D}\le \mathrm{tol}$}
		%		\STATE {$\qquad$ Decrease $\epsilon_{0}$ and $\epsilon_{1}$ using an appropriate strategy}
		%		\ENDIF      
		%	\ENDWHILE
		%\STATE{Recover denoised image $u$ as}
		\STATE{Unless some stopping criterion is satisfied, stop}
	\end{algorithmic}
\end{algorithm} 
For color images, denoting the color image as $\textbf{x}=(\bds{x}_1,\bds{x}_2, \bds{x}_3)^{T}$, the truncated quadratic regularization models are as follows %\begin{equation}
\begin{align}
&\argmin_{\bf x} \mathcal{F}({\bf x}) = \frac{\|{\bf A}{\bf x} -{\bf x}_0\|_2^2}{2} +
\sum_{i=1}^m \sum_{j=1}^n\frac{\mu}{2}\min(|(\nabla \textbf{x})_{i,j}|^2,  \frac{\lambda}{\mu }),\ \ \text{isotropic case}  \label{eq:vector:color} \\
&\argmin_{\bf x} \mathcal{F}({\bf x}) = \frac{\|{\bf A}{\bf x} -{\bf x}_0\|_2^2}{2} +\sum_{i=1}^m \sum_{j=1}^n \sum_{k=1}^3\frac{\mu}{2}\min(|(\nabla \bds{x}_k)_{i,j}|^2, \frac{\lambda}{\mu }), \ \ i=1,2,3, \ \ \text{anisotropic case}  \notag
\end{align}
%\end{equation}
where $|\nabla {\bf  x}|^2=\sum_{i=1}^3 |\nabla \bds{x}_i|^2$ and $|\nabla \bds{x}_i|^2 = |\nabla_1 \bds{x}_i|^2 + |\nabla_2 \bds{x}_i|^2$, $i=1,2,3$ and $\bf A$ is a linear and bounded operator. It can be seen that the functional of the isotropic case in \eqref{eq:vector:color} is still within the form of Theorem \ref{thm:KL}. For the anisotropic case, denoting ${\bf K}_1 = \text{Diag}[\nabla, 0,0]$, ${\bf K}_2 = \text{Diag}[0, \nabla, 0]$ and ${\bf K}_3 = \text{Diag}[0,0, \nabla]$, then the functional
\[
\sum_{i=1}^3\frac{\mu}{2}\min(|\nabla \bds{x}_i|^2, \frac{\lambda}{\mu }) = \sum_{i=1}^3\frac{\mu}{2}\min(|{\bf K_i}{\bf x}|^2, \frac{\lambda}{\mu }),
\]
is still of the form in Theorem \ref{thm:KL} before the summation over all the pixels as in \eqref{eq:vector:color}. The global convergence and local linear convergence rate also follow. The corresponding algorithm is completely similar to Algorithm \ref{alg:dca:deno:scale} and we omit here.
\section{Numerics}\label{sec:num}
In this section, we will  consider the image denoising and image segmentation problem. 
All experiments are performed
in Matlab 2019a on a 64-bit PC with an Inter(R) Core(TM) i5-6300HQ CPU(2.30Hz) and 12
GB of RAM. 
\subsection{Image Denoising}
We will compare with the well-known total variation (TV) regularization
\begin{equation}\label{eq:TV}
\argmin_{\bds{x} \in X} F(\bds{x}) =  \frac{1}{2}\|A\bds{x}-f\|_{2}^2+ \alpha \|\nabla \bds{x}\|_1, 
\end{equation}
For image denoising, $A=\bds{I}$. The first-order primal-dual algorithm is employed for the minimization problem \eqref{eq:TV} \cite{CP}. We will also compare with the appealing truncated regularization framework developed in  \cite{WLW} including the truncated TV (shorten as TR-TV), truncated logarithmic regularization (shorten as TR-LN), the truncated quadratic regularization (shorten as TR-$l_2$), {and the weighted difference of anisotropic and isotropic total variation mode (Ani-iso-DCA) \cite{LZOX}}. The TR-$l_2$ models are the same as \eqref{equation:iso} and \eqref{equation:ani}. As in \cite{WLW}, ADMM (Alternating direction method of multipliers) type method is employed to solve the TR-TV, TR-$l_2$, and TR-LN.  It is already shown TR-TV and TR-LN can give promising PSNR especially for isotropic cases \cite{WLW}. Here we focus on the anisotropic cases. 
{The extrapolation parameter $\{ \beta_t\}$ is chosen according to \cite{WCP} for the proposed preconditioned DCA, where 
\begin{equation}\label{eq:betat}
\beta_t = (\theta_{t-1}-1)/{\theta_t}, \quad \theta_t = (1+\sqrt{4\theta_{t-1}^2 +1})/{2}, \quad \theta_{-1}=\theta_0=1.
\end{equation}
Restarting strategy is necessary for satisfying the condition $\{\beta_t\} \in [0,1)$ and $\sup_{t}\beta_t <1$.  The adaptive $\beta_t$ in \eqref{eq:betat} can bring out certain acceleration experimentally.}
 With appropriate parameters of $\lambda$ and $\beta$, it can be seen that the truncated regularization \eqref{equation:iso} and \eqref{equation:ani} can obtain high quality denoised images; see Figure \ref{lena:ani:tv:dca} for the anisotropic truncated quadratic case \eqref{equation:ani} and Figure \ref{monarch:itq:denoise} for the isotropic truncated quadratic case \eqref{equation:iso}. Especially, there is no staircasing effect for \eqref{equation:iso} or \eqref{equation:ani} as the total variation. From Figure \ref{lena:psnr:iter:time}, it can be seen that the \eqref{equation:ani} can get better PSNR with less iterations and less computation time compared with the anisotropic TV.
 
From Table \ref{tab:ani:atq:atv} which is focused on the anisotropic cases,  it can be seen that both  \eqref{equation:ani} and TR-LN are very competitive with high PSNR values for most cases compared with TV. The TR-TV can get higher SSIM for some cases. Although the same model with the same parameters $\lambda$ and $\mu$ for \eqref{equation:ani}, our proposed preconditioned DCA can get higher PSNR and SSIM compared with ADMM used in \cite{WLW}. The preconditioned DCA may exploit more potential of the model  \eqref{equation:ani} compared to the ADMM employed in \cite{WLW}. {For the comparison with computational efficiency, Figure \ref{DCA_ADMM_comparison} tells that while the proposed preconditioned DCA can decrease the energy quickly and achieve a better PSNR value much fast compared with both iteration number and iteration time, the ADMM employed in \cite{WLW} can obtain a lower energy with enough iterations.} 
{Tables \ref{tab:ani:atq:atv} also shows that the Ani-iso-DCA \cite{LZOX} is also competitive 
	compared to TV. However, it is not as promising as ATQ and TR-LN models.}

For the global convergence with preconditioners, Figure \ref{global:convergence} tells that the proposed preconditioned DCA is faster than DCA with solving the linear subproblem very accurately by the DCT (Discrete cosine transform) compared both with iteration number and computational time. This is surprising that the proposed preconditioned DCA not only can save the computational efforts but also can improve the performance of DCA with more efficient  algorithms. For the local convergence rate, Figure \ref{local:rate}(a) tells that for the whole nonlinear DCA iterations,  for the linear system appeared, the SRBGS preconditioner is very efficient compared to solving the linear subproblems very accurately with DCT. The proposed preconditioned DCA can get faster local linear convergence rate with less computations compared to the original proximal DCA with highly accurate DCT solver. Theoretically, the proposed preconditioned DCA not only provides an efficient inexact framework with any finite time preconditioned iterations for DCA with global convergence guarantee, but also can potentially give a faster local convergence rate compared to the original DCA with a very accurate solver. 

{Figure \ref{fig:seg} shows that the proposed preconditioned DCA can be used for image segmentation with various examples, which is not surprising since the truncated quadratic model is widely studied and used for image segmentation problems \cite{BVZ, BZ}.  Figure \ref{fig:seg} also shows that the truncated quadratic model can give better segmentation than TV.}

\subsection{Image deblurring}
For image deblurring, by Proposition \ref{pro:feaible:percon:deblur}, we just need to design a preconditioner for the discrete version of the  following equation 
	\begin{equation}\label{eq:deblur:neu}
	T_n \bds{x}  = b^t, \quad T_n: = (L_0\bds{I} +A^*A - \mu \Delta_n ), \quad \frac{\partial \bds{x} }{ \partial \nu} = 0,
	\end{equation}
	where we use $\Delta_n $ to denote the Laplacian operator with emphasis on the Neumann boudary condition. 
	The above equation is usually solved directly by FFT (fast fourier transform). However, considering the FFT is based on the periodic boundary condition which does not match the Neumann boundary condition, it can be circumvented through preconditioning technique \cite{BS2}, 
	\begin{equation}\label{eq:deblur:peo}
	T_p x = b^t, \quad T_p: = (L_0\bds{I} +A^*A - \mu \Delta_p ), \quad x \ \  \text{with periodic boundary condition},
	\end{equation}
	where $\Delta_p $ denotes the Laplacian operator with emphasis on the  periodic boundary condition. It is proved that $T_p \geq T_n$ \cite{BS2}. We can use $T_p$ as a preconditioner for $T_n$ as follows 
	\[
	x^{t+1} = y^t + T_p^{-1}(b^t - T_n y^t)=  T_p^{-1}(b^t + T_p y^t - T_n y^t) = T_p^{-1}(b^t -\mu(\Delta_p-\Delta_n)y^t).
	\]
	Since $T_p = A^*A - \mu \Delta_n + M $ by Proposition \ref{pro:feaible:percon:deblur}, we have the proximal metric $M = T_p - T_n + L_0 \bds{I}  \geq L_0 \bds{I}$. Denoting the periodic
	convolution kernel of $-\Delta$ is $\kappa_{\Delta}$ along with $\mathcal{F}$ and $\mathcal{F}^{-1}$ being the discrete Fourier and inverse Fourier transform \cite{BS2}, with these preparations, we now give our Algorithm \ref{alg:dca:deblur} for image deblurring.
	
	\begin{algorithm}
		\caption{preDCA$_{e}$ for image deblurring of the truncated model \eqref{equation:iso} or \eqref{equation:ani} with $Ax = x \ast \kappa $ and $A^*x = x \ast \kappa'  $ with $\kappa$ being the convolution kernel \label{alg:dca:deblur}}
		\begin{algorithmic}
			\STATE { $x^{0} \in \operatorname{dom} P_{1}$, $\{\beta_{t}\}  \subseteq [0,1)$, \text { with } $\sup _{t} \beta_{t}<1$. Choose $L_0> 0$ and set  $x^{-1}=x^{0}$. \\
				Iterate the following steps for $t=0, 1, \cdots$,
				%				\STATE {
				%					\begin{equation}
				%	
				%					\end{equation}	}
				\begin{align}
				\xi^t& = \begin{cases}
				\nabla^*\chi_{x^t}\nabla x^t, \quad \text{for the isotropic case,} \\ 
				(\nabla_1^*\chi_{x^t,1}\nabla_1 + \nabla_2^*\chi_{x^t,2}\nabla_2) x^t, \quad \text{for the anisotropic case,} 
				\end{cases}\\
				y^{t}&= x^{t}+\beta_{t}(x^{t}-x^{t-1}), \notag
				\\
				b^t &= L_0y^t + x_0 \ast \kappa' + \xi^t,  \\
				x^{t+1} & \displaystyle = \mF^{-1}\Bigl(
				\frac{\mF( b^t) }
				{|\mF (\kappa)|^2 + \mu \mF (\kappa_\Delta) +L_0} \Bigr).
				%	x^{t+1} &= \SRBGS_{\alpha,\beta}^n(y^t, b^t), \quad T: = (L_0I -\mu\Delta).:pre} %\label{eq:proxi:y}
				\end{align}
				
			}
			%\STATE {$\qquad\left (\epsilon \Delta^{2}+(x^{2})^{\ast}x^{2}+\frac{1}{\epsilon_{0}} \big (Df_{0}(p)+ Dg_{0}(p)  \big ) +\frac{1}{\epsilon_{1}} \big ( D\hat{f}_{1}(p)+D\hat{g}_{1}(p\big) \right )d=b$}
			%		\STATE{\textbf{Step 2: }Update $h^{l+1}$ by \eqref{eq:ufirst:h}} (or \eqref{eq:ufirst:h:ani} for the anisotropic case)
			%		\STATE{\textbf{Step 3: }Project $h^{l+1}$ to the feasible set  $\{h:\norm[\infty]{h} \leq \alpha\}$, i.e., $h^{l+1}  =\mathcal{P}_{\alpha}(h^{l+1})$. Set $(u^{l+1},h^{l+1})$ as the initial value for the next Newton iteration and go to \textbf{Step 1}}.
			%		\IF{$\|F(p)\|_{D}\le \mathrm{tol}$}
			%		\STATE {$\qquad$ Decrease $\epsilon_{0}$ and $\epsilon_{1}$ using an appropriate strategy}
			%		\ENDIF      
			%	\ENDWHILE
			%\STATE{Recover denoised image $u$ as}
			\STATE{Unless some stopping criterion is satisfied, stop}
		\end{algorithmic}
	\end{algorithm}

	For image deblurring of anisotropic cases, we will compare with TV, i.e., $Au= u \ast \kappa $ in \eqref{eq:TV} with first-order primal-dual algorithm \cite{CP}, the  TR-TV and TR-$l_2$ models in \cite{WLW} with ADMM who can get stable PSNR during iterations. We also compared  with the DCA without preconditioning  by $x^{t+1} = T_p^{-1}(b^t)$ and we denote it as ATQ-Npre. In ATQ-Npre, the different boundary conditions of $T_p$ and $T_n$ are ignored, and FFT together with inverse FFT is directly applied to the Neumann boundary condition $T_n$. 
	
	Figure \ref{llama:itq:deburring} tells that we can get high quality deblurred images with \eqref{equation:ani} with our preconditioned DCA, i.e., Algorithm \ref{alg:dca:deblur} for degraded images blurred by motion filter or Gaussian filter. Here we choose $L_0=10^{-10}$ to approximate the original linear system \eqref{eq:linear:DCA:ori} of the standard DCA. 
	Table \ref{tab:ani:deblur} shows that \eqref{equation:ani} with the proposed algorithm can get competitive PSNR and SSIM. Our preconditioned DCA can still obtain better PSNR or SSIM compared the TR-$l_2$ by ADMM. Both Table \ref{tab:ani:deblur} and Figure \ref{local:rate}(b) shows that \eqref{equation:ani} with preconditioned DCA in Algorithm \ref{alg:dca:deblur} can get better PSNR, SSIM and lower energy compared to the DCA without preconditioning, i.e., ATQ-Npre. {The performance of Ani-iso-DCA \cite{LZOX} is similar to the denoising case, which is competitive compared to TV.}

	We also found that Algorithm \ref{alg:dca:deblur} with small $L_0$ can get much better PSNR and SSIM than Algorithm \ref{alg:pre:dca} for imaging deblurring where the $A^*A$ is put into the backward step. Since whose PSNR and SSIM are much lower according to our experience, we did not present the corresponding numerical results.

\begin{table}%\label{tab:ani:atq:atv}
	\centering
	\caption{Comparison for anisotropic image denoising models. The noisy images are as follows:	Lena1, Lena2 with size $512 \times 512$, Monarch1 and Monarch2 with size $768 \times 512$. The usual zero mean Gaussian white noise of variance $\sigma =0.1$ for Lena1 or Monarch1 and $\sigma=0.05$ for Lena2 or Monarch2. The  parameters for the corresponding models are as follows. For \eqref{equation:ani}, we choose {$\mu=3,\lambda=0.01$} for $\sigma=0.1$ cases and {$\mu=1.5,\lambda=0.005$} for $\sigma=0.05$ cases. For the anisotropic TV model, the regularization parameter $\alpha$ is chosen as the variance of the noise, i.e., $\alpha = \sigma$. For the truncated regularization, the  parameters of TR-TV are $\alpha=10,\beta=600,\tau=0.6$ for $\sigma=0.1$ cases and  $\alpha=40,\beta=6000,\tau=0.2$ for $\sigma=.05$ cases. 
		The parameters of TR-$l_2$  are {$\alpha=2/3,\beta=6000,\tau=0.0577$}  for $\sigma=0.1$ cases and {$\alpha=4/3,\beta=6000,\tau=0.0577$} for $\sigma=0.05$ cases. The  parameters of TR-LN are $\alpha=10,\beta=600,\tau=0.5,\theta=1$ for $\sigma=0.1$ cases and $\alpha=40,\beta=600,\tau=0.5,\theta=1$ for $\sigma=0.05$ cases. {The parameters of Ani-iso-DCA are $\mu = 5,\lambda=0.5$ for $\sigma=0.1$ cases and $\mu=15,\lambda=1$ for $\sigma=0.05$ cases.}}
	\label{tab:ani:atq:atv}
	\begin{tabular}{|c|c|c|c|c|c|c|c|c|}
		\hline
		\multirow{2}{*}{ }	& \multicolumn{2}{c|}{ Lena1} & \multicolumn{2}{c|}{Monarch1 }&\multicolumn{2}{c|}{Lena2  }& \multicolumn{2}{c|}{Monarch2} \\
		\hline 
		& PSNR & SSIM & PSNR & SSIM & PSNR & SSIM&PSNR &SSIM\\
		\hline
		ATQ model~&\bf 29.308& 0.784&\bf 29.620&0.836&32.380&\bf 0.859&33.203&0.898 \\
		\hline
		TV model~ & 29.227&0.800&29.143&0.873&31.850&0.854&32.613&\bf 0.919 \\
		\hline
		TR-TV~ & 29.250&\bf 0.801& 29.169&\bf 0.875&32.601&0.854&33.178&0.893\\
		\hline
		TR-$l_2$~& 29.079&0.741&27.851&0.768&31.950&0.832&30.975&0.853\\
		\hline
		TR-LN~&29.285&0.800&29.361&0.873&\bf 32.615&0.850&\bf 33.223&0.887\\
		\hline
		Ani-iso-DCA~&29.113&0.793&29.227&0.864&32.293&\bf 0.859&33.000&0.909\\
		\hline 
	\end{tabular}
\end{table}

\begin{figure}%[!htb]
	%\graphicspath{{fig//}}
	\begin{center}
		\subfloat[Noisy image with $\sigma=0.1$]
		{\includegraphics[width=4.5cm]{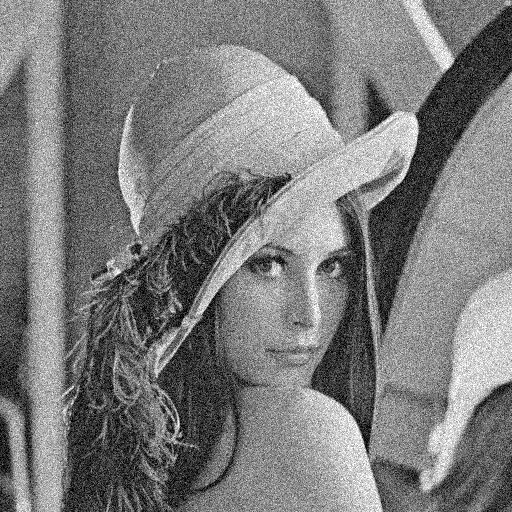}} \ \ 
		\subfloat[ATQ model with $\lambda=3,\mu=0.01$]
		{\includegraphics[width=4.5cm]{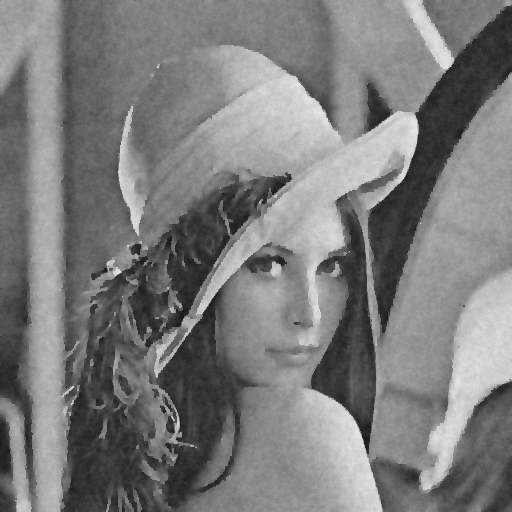}}\ \
		\subfloat[TV model, $\alpha=0.1$]
		{\includegraphics[width=4.5cm]{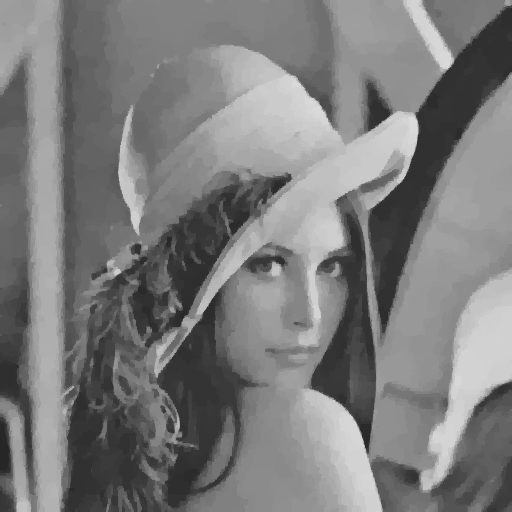}}\\
		\subfloat[Noisy image with $\sigma=0.05$]
		{\includegraphics[width=4.5cm]{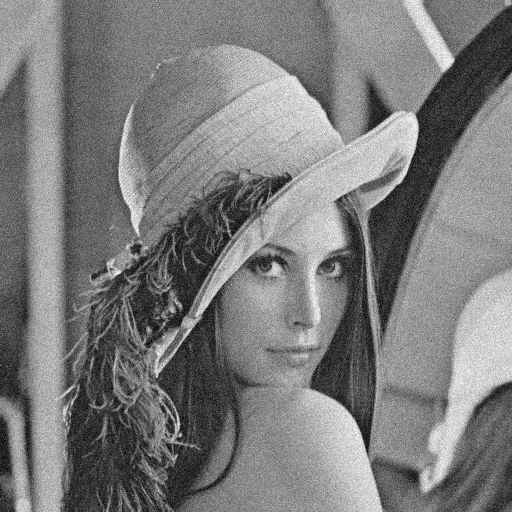}}\ \
		\subfloat[ATQ model with $\lambda=1.5,\mu=0.005$]
		{\includegraphics[width=4.5cm]{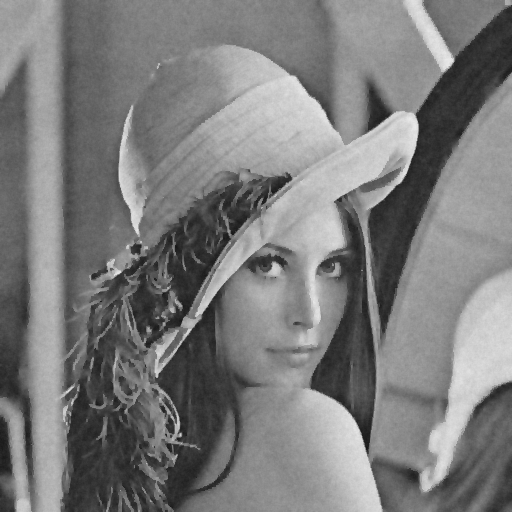}}\ \
		\subfloat[TV model, $\alpha=0.05$]
		{\includegraphics[width=4.5cm]{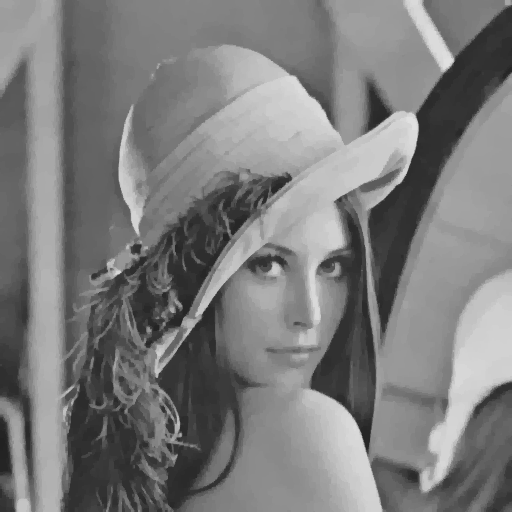}}
	\end{center}
	%\caption{Fronalpstock}
	\caption{\scriptsize{Images (a) and (d) show the corresponding noisy {images} of the standard {Lena} image of size $512 \times 512$ corrupted by Gaussian noise of Gaussian variance $\sigma=0.1$ and $\sigma=0.05$, i.e., Lena1 and Lena2 image in Table \ref{tab:ani:atq:atv}. Images (b) and  (e) show the denoised images of (a) and (d) with ATQ by parameters {$\mu=3,\lambda=0.01$} and  {$\mu=1.5,\lambda=0.005$} correspondingly. Images (c) and (f) are denoised images of (a) and (d) by the anisotropic TV through the first-order primal-dual algorithm with the corresponding parameters $\alpha = 0.1$ and $\alpha=0.05$.}}
	\label{lena:ani:tv:dca}
\end{figure}
\begin{figure}%[!htb]
	%\graphicspath{{fig//}}
	\begin{center}
		\subfloat[Original image: Tucan]
		{\includegraphics[width=4.5cm]{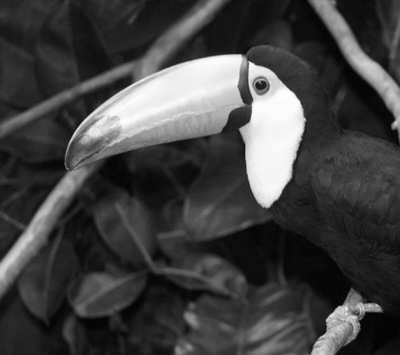}}\ \ 
		\subfloat[Noisy image with $\sigma=0.1$]
		{\includegraphics[width=4.5cm]{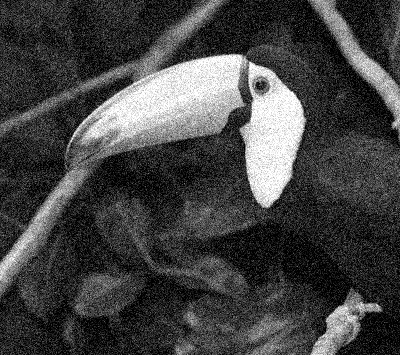}} \ \ 
		\subfloat[ITQ model with $\lambda=3,\mu=0.01$]
		{\includegraphics[width=4.5cm]{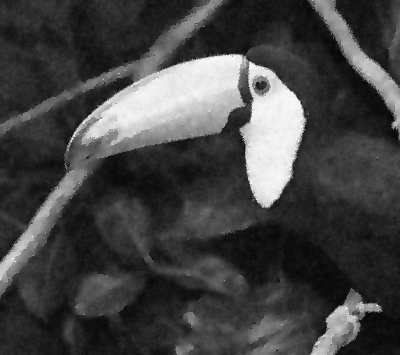}}\ \ 
	\end{center}
	%\caption{Fronalpstock}
	\caption{\scriptsize{Images (a) shows the original $400 \times 355$ Tucan  image. Image (b) is a noisy image corrupted by 10\% Gaussian noise. Image} (c)  shows the denoised image by \eqref{equation:iso} with parameters  $\mu=3,\lambda=0.01$.} 
	\label{monarch:itq:denoise}
\end{figure}

\begin{figure}%[!htb]
	%\graphicspath{{fig//}}
	\begin{center}
		\subfloat[PSNR comparison with iteration number between ATQ and anisotropic TV]
		{\includegraphics[width=6.5cm]{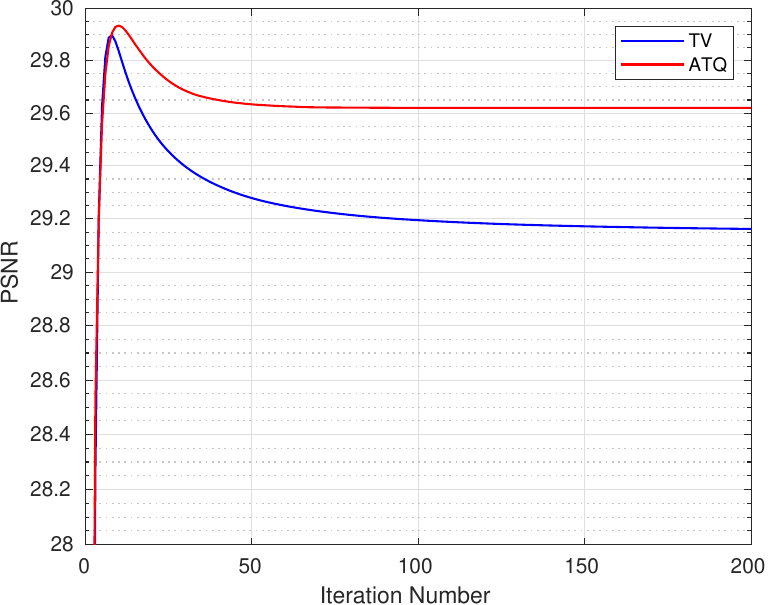}}\quad
		\subfloat[PSNR comparison with iteration time between ATQ and anisotropic TV]
		{\includegraphics[width=6.5cm]{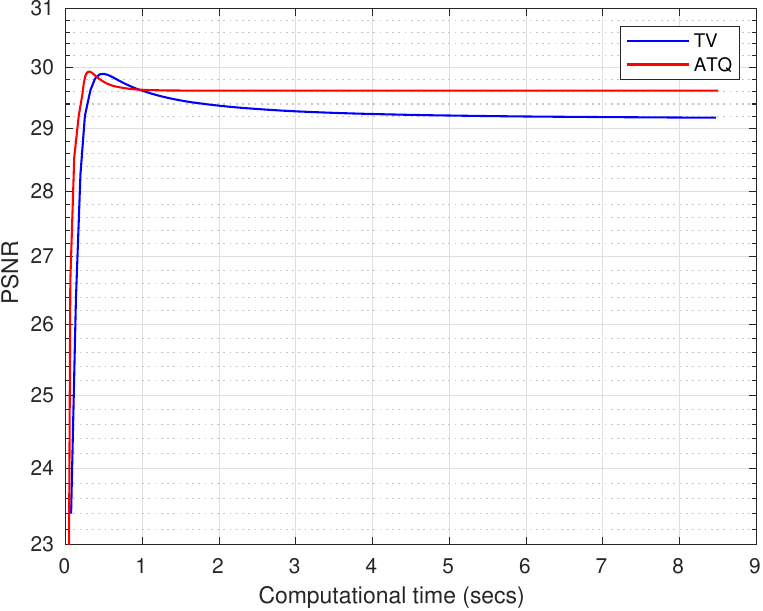}} 
	\end{center}
	\caption{\scriptsize{Figures (a) or (b) shows the PSNR comparisons with iteration number or computational time between \eqref{equation:ani} and the anisotropic TV. The computations are based on the Monarch image of size $768\times 512$. The  parameters of \eqref{equation:ani} are {$\mu=3,\lambda=0.01$} and the parameter of the anisotropic TV is $\alpha=0.1$.} }
	\label{lena:psnr:iter:time}
\end{figure}

\begin{figure}%[!htb]
	%\graphicspath{{fig//}}
	\begin{center}
		\subfloat[Comparison with iteration number between SRBGS preconditioner and the DCT slover]
		{\includegraphics[width=6.5cm]{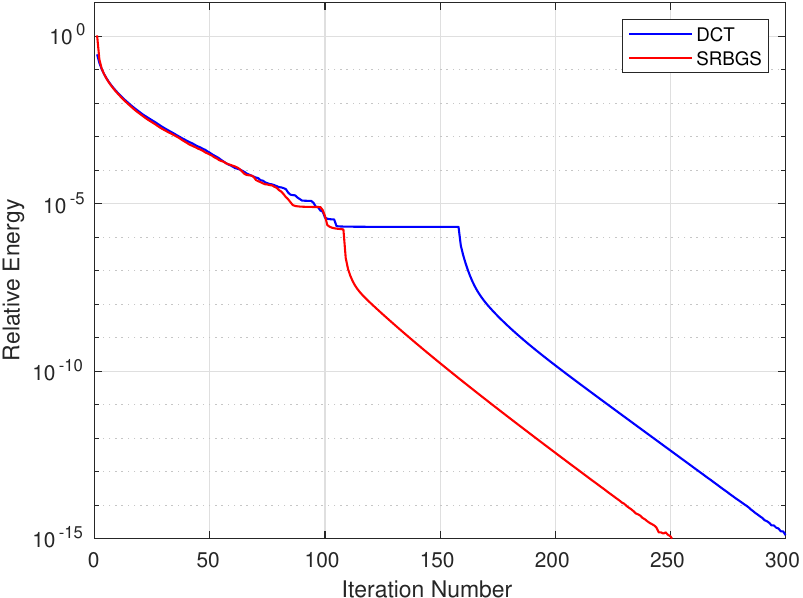}}\quad
		\subfloat[Comparison with iteration time between SRBGS preconditioner and the DCT slover]
		{\includegraphics[width=6.3cm]{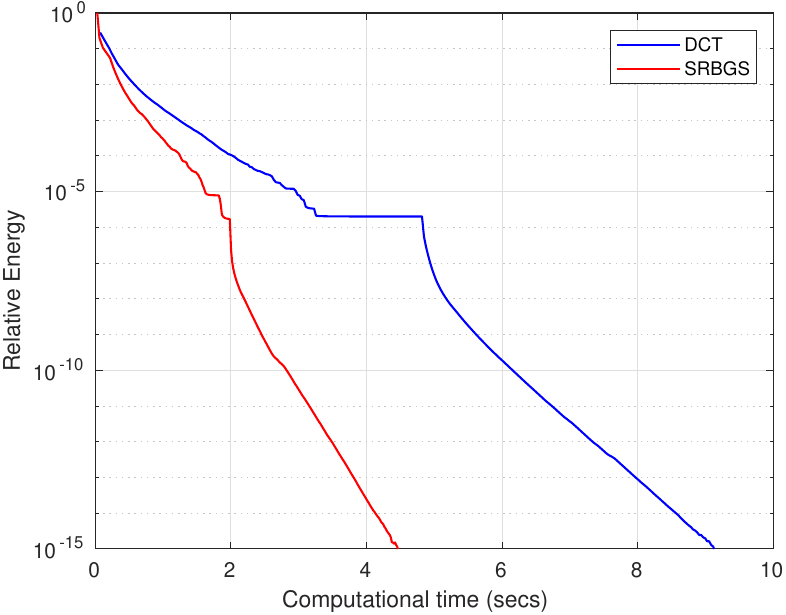}} 
	\end{center}
	\caption{\scriptsize{Figures (a) or (b) shows the comparison with iteration number or computational time between preconditioned DCA with 10 times symmetric Red-Black Gauss-Seidel (SRBGS) iterations  and the DCT solver. The DCT solver can be seen as an approximately exact solver without preconditioners, i.e., $M=L\bds{I}$. The computation is based on Monarch with size $768\times512$ for the model \eqref{equation:ani} with parameters $\mu=3,\lambda=0.01$. }}
	\label{global:convergence}
\end{figure}

%And our PI model can also work.

\begin{figure}[!htbp]
	\centering
		\subfloat[The local linear convergence rate for denoising]
	{\includegraphics[width=0.45\textwidth]{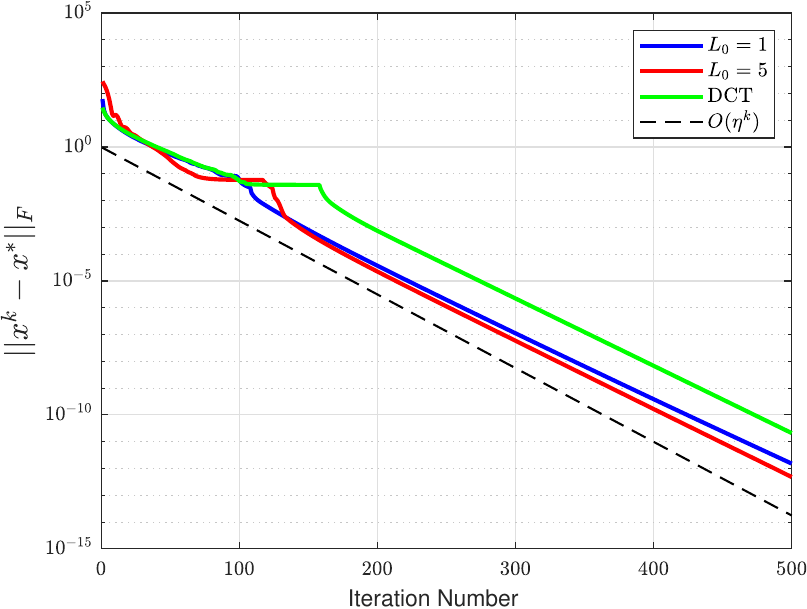}} \quad 
		\subfloat[Energy comparison for deblurring]
		{\includegraphics[width=0.45\textwidth]{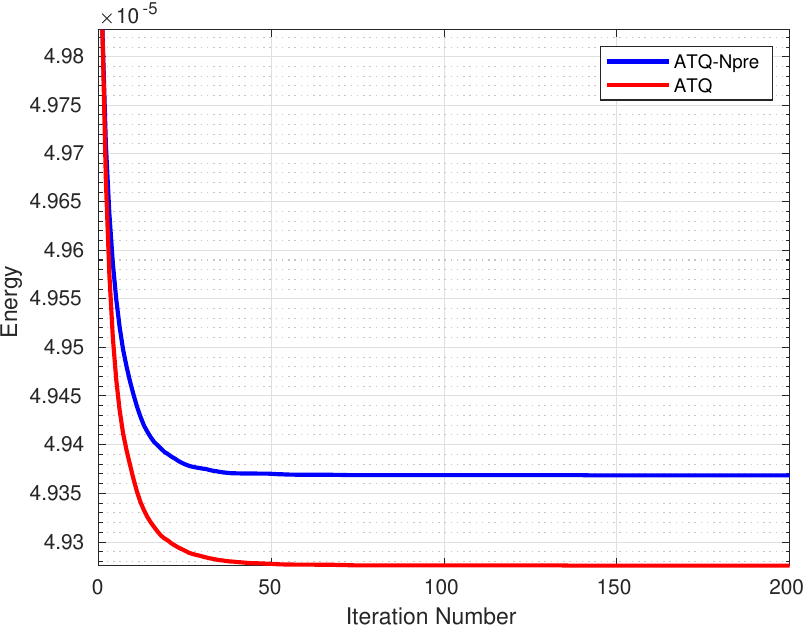}}
	\caption{Figure (a): The local linear convergence rate. The computation is by the model \eqref{equation:ani} for the Monarch image of size $768\times 512$ with parameters $\mu=3,\lambda=0.01$. $L_0$ is as in Lemma \ref{lem:feaible:percon} and DCT represents the case $M=LI$ without preconditoner and solving the corresponding linear equation with DCT solver. The preconditioned DCA for different $L_0$ are both with 10 times symmetric Red-Black Gauss-Seidel (SRBGS) iterations. $\eta \in (0,1)$ is a constant. Figure (b): The energy of precondition ATQ is lower that of ATQ without preconditioning (ATQ-Npre) for deblurring the Llama1 image with Motion filter blur and parameters $\mu = 0.01,\lambda=10^{-4}$ as in Table \ref{tab:ani:deblur}. }
	\label{local:rate}
\end{figure}

\begin{figure}%[!htb]
	%\graphicspath{{fig//}}
%	\begin{center}
%		\subfloat[Energy comparison with iteration number between ATQ and anisotropic ADMM]
%		{\includegraphics[width=6.5cm]{fig/energy_DCA_ADMM_iter.eps}}\quad
%		\subfloat[Energy comparison with iteration time between ATQ and anisotropic ADMM]
%		{\includegraphics[width=6.5cm]{fig/energy_time_DCA_ADMM.eps}} \\
%		\subfloat[PSNR comparison with iteration number between ATQ and anistropic ADMM]
%		{\includegraphics[width=6.5cm]{fig/psnr_DCA_ADMM.eps}}\quad
%		\subfloat[PSNR comparison with iteration time between ATQ and anistropic ADMM]
%		{\includegraphics[width=6.5cm]{fig/psnr_time_DCA_ADMM.eps}}
%	\end{center}

	\begin{center}
	\subfloat[Energy comparison with iteration number between ATQ and anisotropic ADMM]
	{\includegraphics[width=6.5cm]{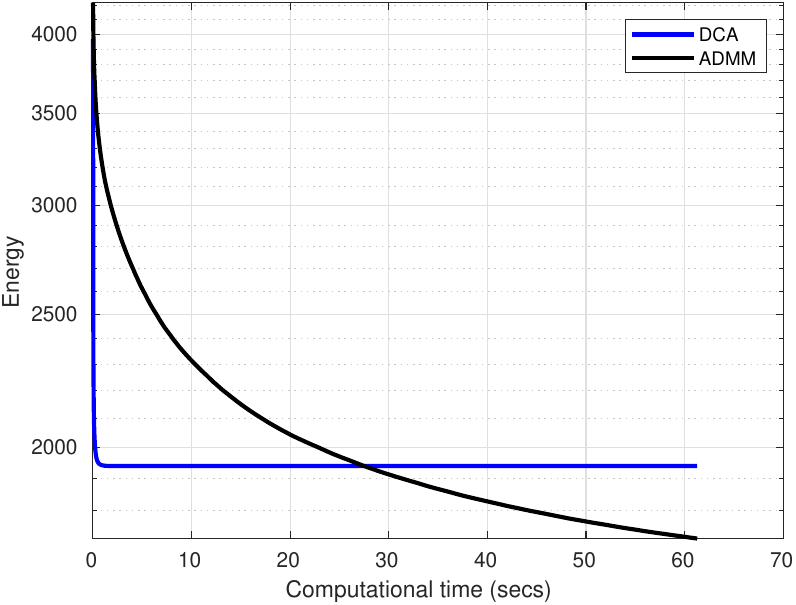}}\quad
	\subfloat[Energy comparison with iteration time between ATQ and anisotropic ADMM]
	{\includegraphics[width=6.5cm]{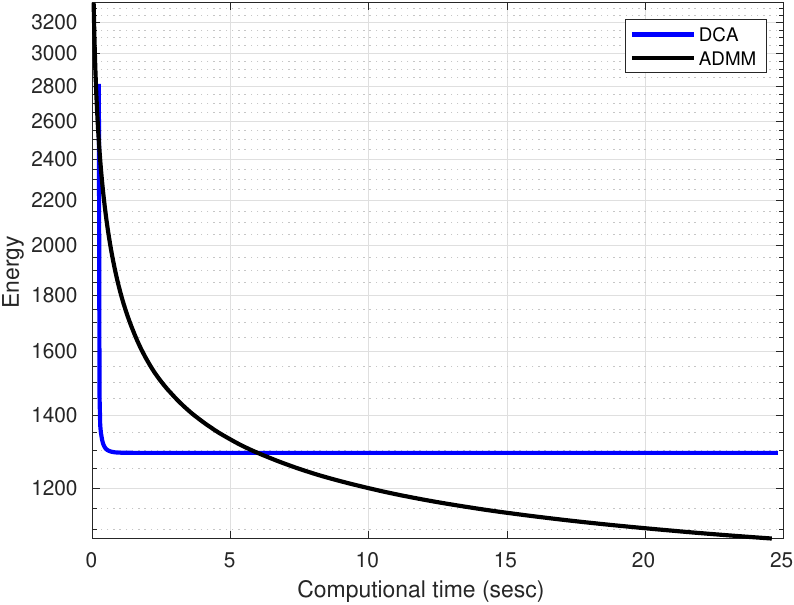}} \\
	\subfloat[PSNR comparison with iteration number between ATQ and anistropic ADMM]
	{\includegraphics[width=6.5cm]{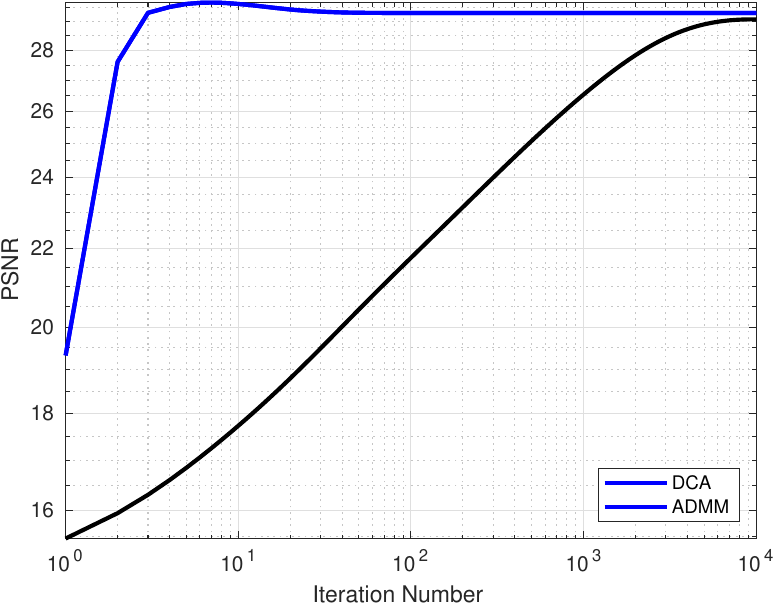}}\quad
	\subfloat[PSNR comparison with iteration time between ATQ and anistropic ADMM]
	{\includegraphics[width=6.5cm]{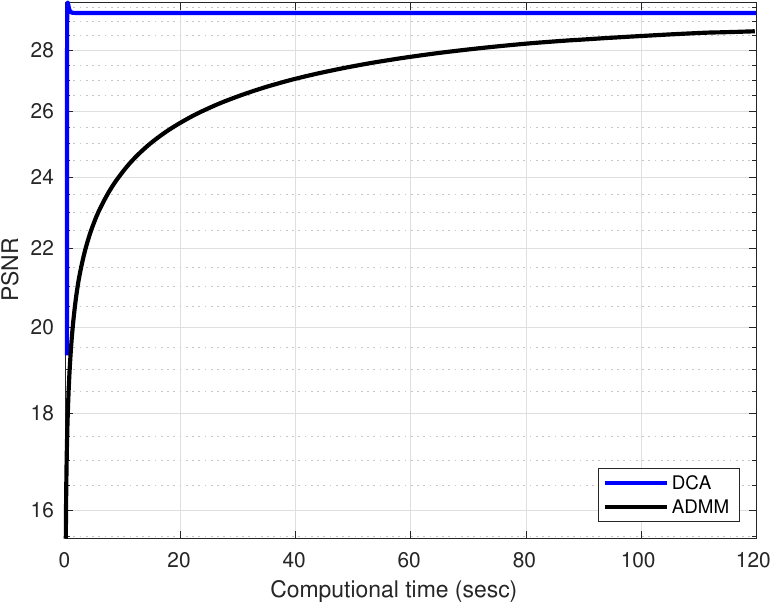}}
\end{center}
		
	\caption{\scriptsize{Figures (a) or (b) shows the energy comparisons with iteration number or computational time between \eqref{equation:ani} and the anisotropic ADMM for image denoising. Figures (c) or (d) shows the PSNR comparison with iteration number or computational time between \eqref{equation:ani} and the anisotropic ADMM for image denoise. The computations are based on the Lena image of size $512\times 512$ and the zeros mean Gaussian white noise of variance $\sigma=0.1$ are added to the image.  The  parameters of \eqref{equation:ani} are $\mu=3,\lambda=0.01$ and the parameters of the anisotropic ADMM is $\alpha=2/3,\beta=6000,\tau=0.0577$.} }
	\label{DCA_ADMM_comparison}
\end{figure}
%\begin{figure}[!htbp]
%	\centering
%	\includegraphics[width=0.5\textwidth]{fig/energy200-eps-converted-to}
%	\caption{The energy of precondition ATQ is lower that nonprecondition ATQ for deblurring question. }
%	\label{pre:better}
%\end{figure}

\begin{figure}%[!htb]
	%\graphicspath{{fig//}}
	\begin{center}
		\subfloat[Original Shooter image]
		{\includegraphics[width=3cm]{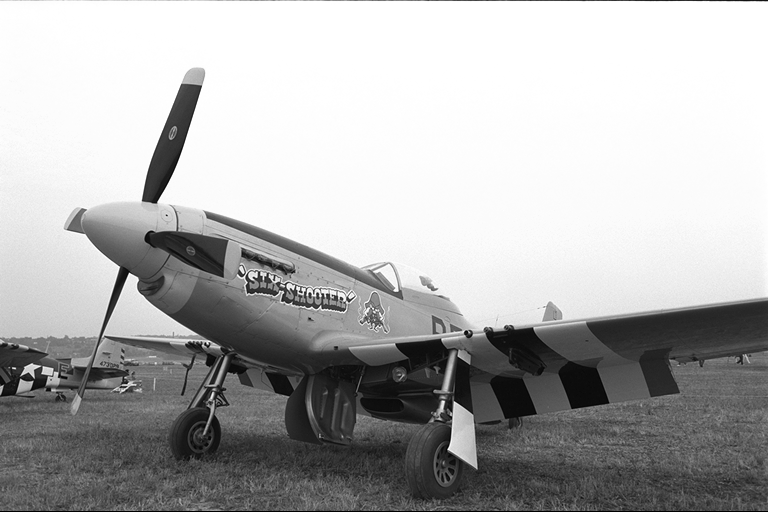}}\ \ 
		\subfloat[ITQ model with $\mu=100,\lambda=0.01$]
		{\includegraphics[width=3cm]{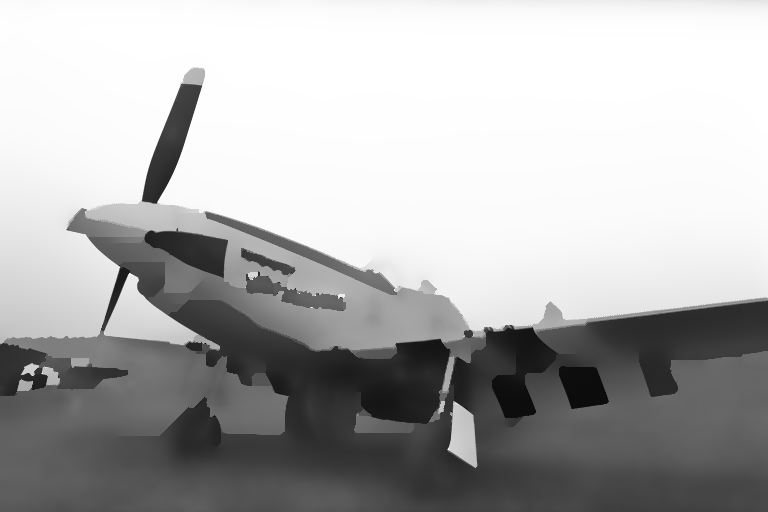}} \ \ 
		\subfloat[ATQ model with $\mu=100,\lambda=0.01$]
		{\includegraphics[width=3cm]{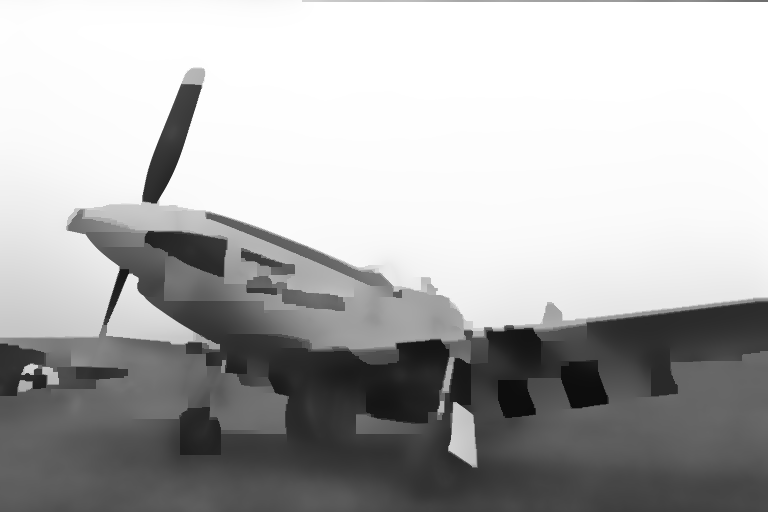}}\ \
		\subfloat[TV model with $\alpha=0.5$]
		{\includegraphics[width=3cm]{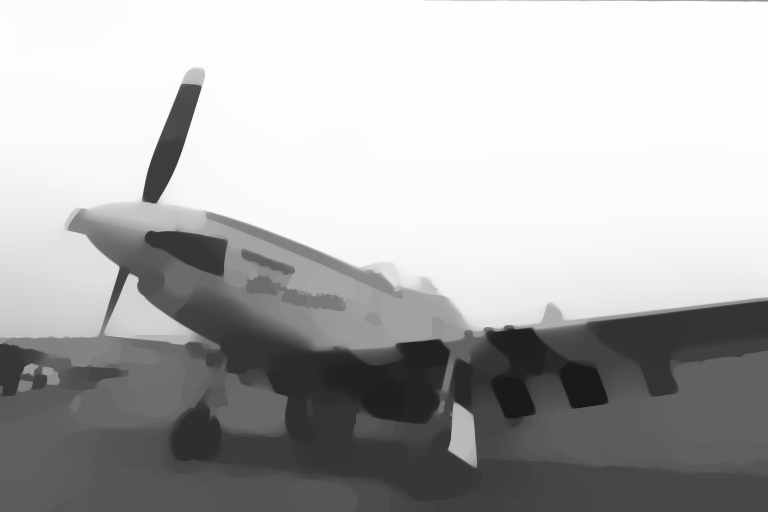}}
		 \\
		\subfloat[Original Monarch image]
		{\includegraphics[width=3cm]{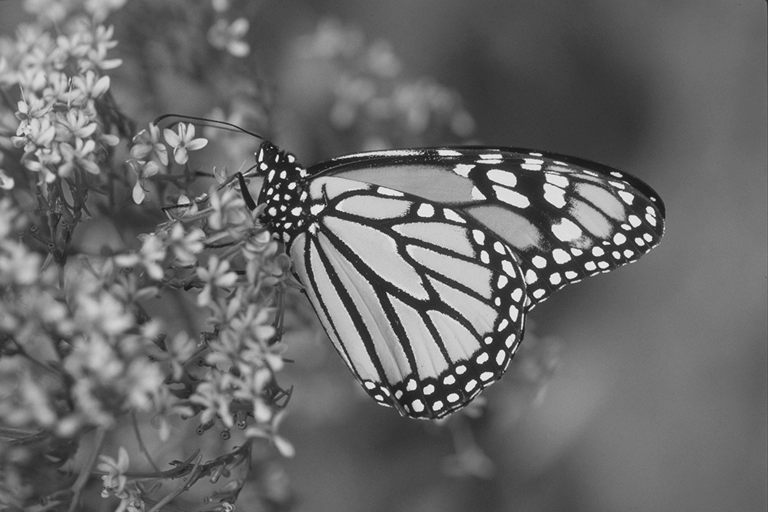}}\ \ 
		\subfloat[ITQ model with $\mu=100,\lambda=0.01$]
		{\includegraphics[width=3cm]{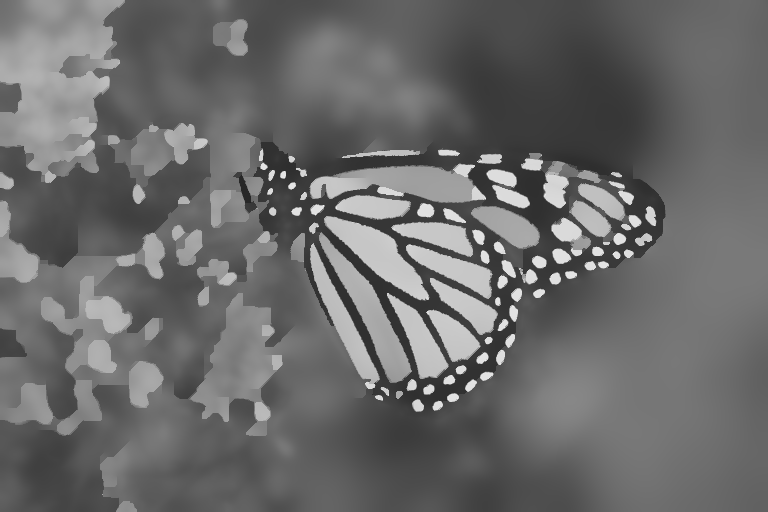}}\ \ 
		\subfloat[ATQ model with $\mu=100,\lambda=0.01$]
		{\includegraphics[width=3cm]{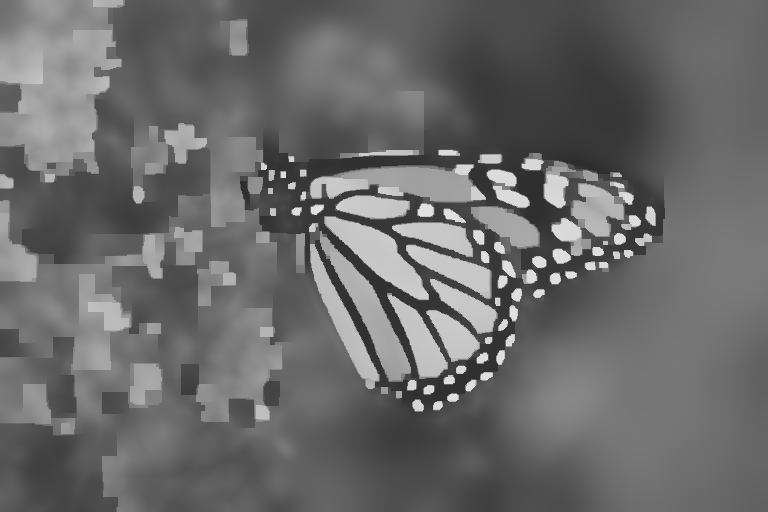}}\ \
		\subfloat[TV model with $\alpha=0.5$]
		{\includegraphics[width=3cm]{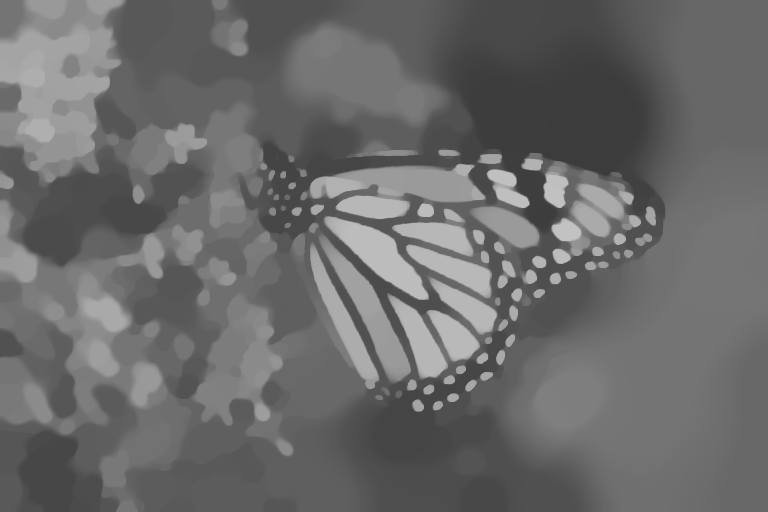}}
		\\
		\subfloat[Orignal Flowers color image]
		{\includegraphics[width=3cm]{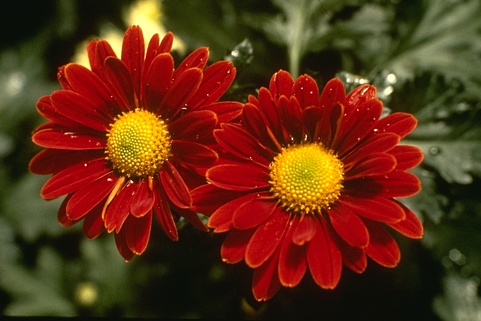}}\ \ 
		\subfloat[ITQ model with $\mu=500,\lambda=0.05$]
		{\includegraphics[width=3cm]{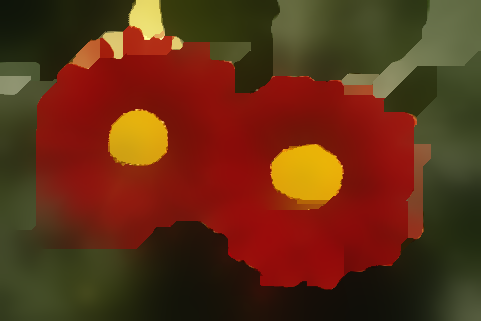}}\ \ 
		\subfloat[ATQ model with $\mu=500,\lambda=0.05$]
		{\includegraphics[width=3cm]{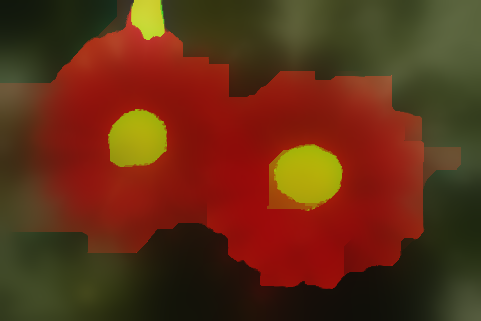}}\ \
		\subfloat[TV model with $\alpha=2$]
		{\includegraphics[width=3cm]{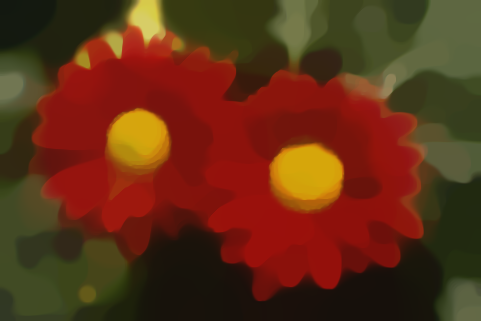}}
		\\
		\subfloat[Original Starfish color image]
		{\includegraphics[width=3cm]{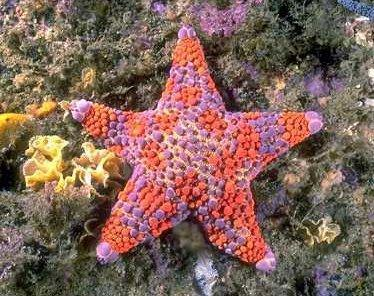}}\ \ 
		\subfloat[ITQ model with $\mu=300,\lambda=0.05$]
		{\includegraphics[width=3cm]{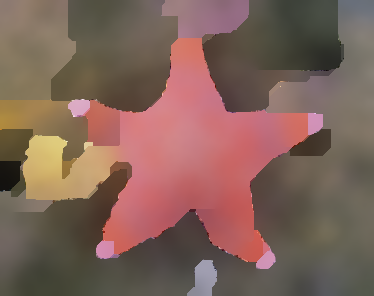}}\ \ 
		\subfloat[ATQ model with $\mu=300,\lambda=0.05$]
		{\includegraphics[width=3cm]{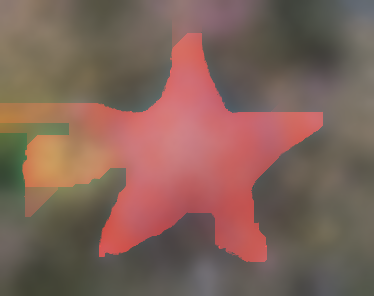}}\ \
		\subfloat[TV model with $\alpha=2$]
		{\includegraphics[width=3cm]{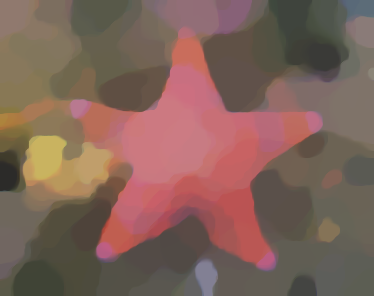}}
	\end{center}
	%\caption{Fronalpstock}
	\caption{\scriptsize{Images (a), (d), (g), (j) show the original gray image Shooter with size $768\times 512$, gray image Monarch with size $768\times 512$, the color image Flowers with size $482\times321$,  the color image Starfish with size  $374 \times 296 $ respectively.  The images in the middle column are segmented by \eqref{equation:iso} while the images in the right column are segmented by \eqref{equation:ani}. Note that the parameters $\mu$ and $\lambda$ can be different for \eqref{equation:iso} and \eqref{equation:ani} even for the same image, since we choose the best parameter as we can find.} }
	\label{fig:seg}
\end{figure}

\begin{table}%\label{tab:ani:atq:atv}
	\centering
	\caption{Comparison for anisotropic image deblurring models. The first columns are different degraded images.  The Kodim251 and Llama1 are both  degraded with Gaussian filter with size $11 \times 11$ and Gaussian noise of variance $\sigma=.01$. The Kodim252 and Llama2 are both degraded with motion filter with size $40\times 50$ and Gaussian noise of variance $\sigma=.01$. The corresponding parameters are as follows. %We choose $\mu=0.01$ and $\lambda=10^{-4}$ for both ATQ or ATQ-Npre, $\alpha=\sigma=0.02$ for TV,
		We choose $\mu=0.01$ and $\lambda=10^{-4}$ for both ATQ or ATQ-Npre, $\alpha=10^{-3}$ for TV,
		$\alpha=2000$, $\beta=6000$, $\tau=0.5$ for TR-TV \cite{WLW} which turns out better than $\beta=600$ as in \cite{WLW},  $\alpha=400,\beta=500,\tau=0.1$ for TR-$l_2$ \cite{WLW}, {$\mu=2,{\lambda=5\times 10^{-3}}$ for Kodim251 and Llama1 cases, and {$\mu=5\times10^{-2}$, $\lambda=5 \times 10^{-5}$} for Kodim252 and Llama2 cases}. {The parameters of Ani-iso-DCA are $\mu = 5\times 10^{-2},\lambda=5\times 10^{-5}$ for Gaussian filter and $\mu=2,\lambda=5\times 10^{-3}$ for motion filter.} The Kodiam25 image is taken from \url{http://www.cs.albany.edu/~xypan/research/snr/Kodak.html}.	}
		
		\label{tab:ani:deblur}
		\begin{tabular}{|c|c|c|c|c|c|c|c|c|}
			\hline
			\multirow{2}{*}{ }	& \multicolumn{2}{c|}{ Kodim251} & \multicolumn{2}{c|}{Llama1 }& \multicolumn{2}{c|}{ Kodim252 }&\multicolumn{2}{c|}{Llama2} \\
			\hline 
			& PSNR & SSIM & PSNR & SSIM & PSNR & SSIM&PSNR &SSIM\\
			\hline
			ATQ model~&\bf 24.211&\bf 0.602&\bf 27.771&\bf 0.750&\bf 23.943&0.603&\bf 26.820&\bf 0.712\\
			\hline
			TV model~ &23.655&0.537&26.960&0.704&23.480&0.537&26.081&0.669  \\
			\hline
			TR-TV~&23.959&0.571&27.371&0.728&23.936&0.583&26.643&0.700\\
			\hline
			TR-$l_2$~  &24.112&0.600&27.284&0.707&23.919&\bf 0.604&26.450&0.671\\
			\hline
			ATQ-Npre~&24.134&0.599&27.654&0.747&23.835&0.600&26.758&0.710\\
			\hline
			Ani-iso-DCA~&21.832&0.406&24.349&0.562&20.685&0.356&22.962&0.514\\
			\hline
		\end{tabular}
	\end{table}

\begin{figure}%[!htb]
	%\graphicspath{{fig//}}
	\begin{center}
		\subfloat[Original image: llama]
		{\includegraphics[width=4.5cm]{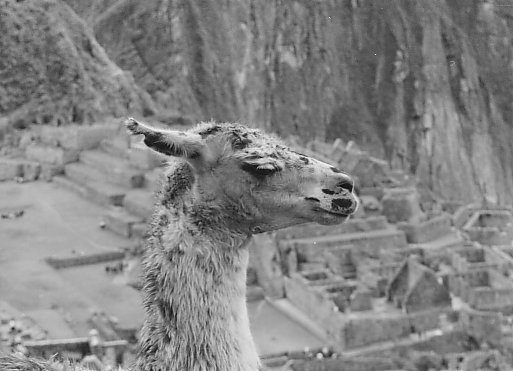}}\ \ 
		\subfloat[Degraded image: motion filter, $\sigma=0.01$]
		{\includegraphics[width=4.5cm]{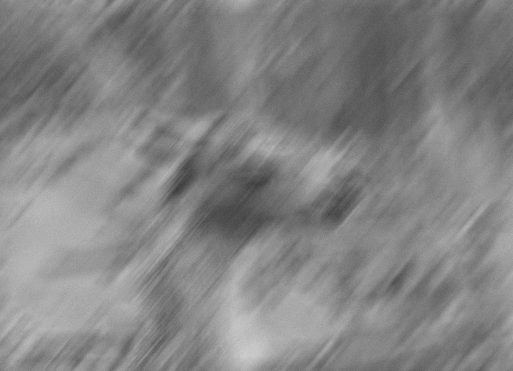}} \ \ 
		\subfloat[ATQ model: $\lambda=0.01$, $\mu=10^{-4}$]
		{\includegraphics[width=4.5cm]{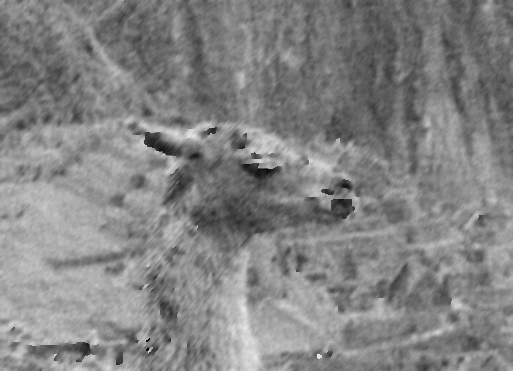}}\ \ 
		\subfloat[Original image: kodim25]
		{\includegraphics[width=4.5cm]{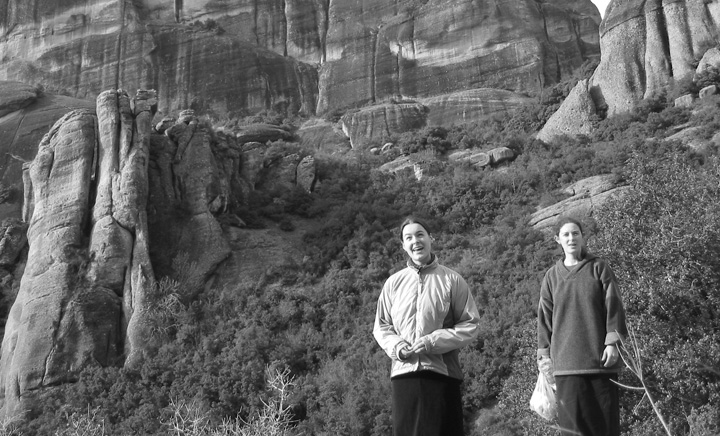}}\ \ 
		\subfloat[Degraded image: Gaussian filter, $\sigma=0.01$]
		{\includegraphics[width=4.5cm]{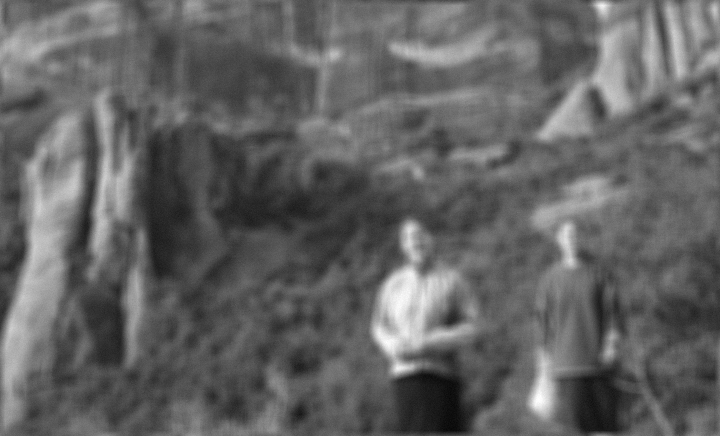}} \ \ 
		\subfloat[ATQ model: $\lambda=0.01$, $\mu=10^{-4}$]
		{\includegraphics[width=4.5cm]{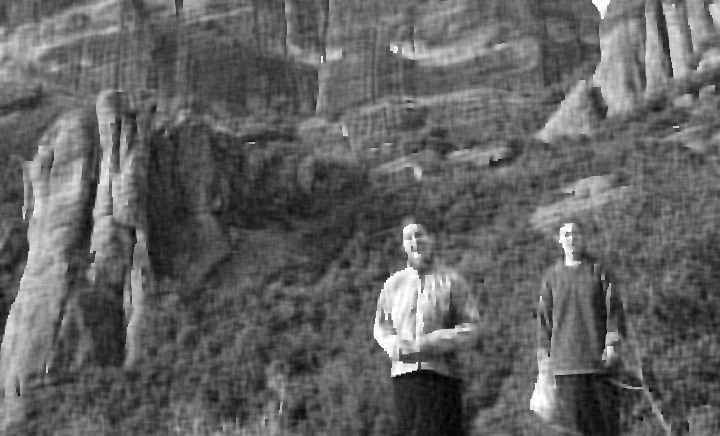}}\ \ 
	\end{center}
	%\caption{Fronalpstock}
	\caption{\scriptsize{Images (a) or (d) show the original $513 \times 371$ Llama image and $720\times 436$ Kodim25 image. Image (b) is degraded Llama2 image by motion blur as in Table \ref{tab:ani:deblur}. Images (c)  shows the reconstructed image by \eqref{equation:ani} with parameters  $\mu=0.01,\lambda=10^{-4}$. Image (e) is degraded Kodim251 image with Guassian filter blur as in Table \ref{tab:ani:deblur}. Images (f)  shows the reconstructed image by \eqref{equation:ani} with parameters  $\mu=0.01,\lambda=10^{-4}$.}}
	\label{llama:itq:deburring}
\end{figure}
%\subsection{Image deblurring}

\section{Discussion and Conclusions}\label{sec:conclude}
In this paper, we give a thorough study on the proposed preconditioned DCA with extrapolation. We analysis it through the proximal DCA with metric proximal terms. We show that our framework is very efficient to deal with linear systems, while the global convergence and the local convergence rate can also be obtained. Numerical results show that the proposed preconditioned DCA is very efficient for truncated regularization applying to image denoising and image segmentation. We will consider other challenging tasks or applications with our preconditioned DCA framework.

\noindent
{\small
	\textbf{Acknowledgements}
	H. Sun acknowledges the support of
	NSF of China under grant No. \,11701563.	{The authors also would like to thank all the anonymous referees for their detailed comments that helped us to improve the manuscript.}
}

%\noindent
%{\small
%	\textbf{Declarations}
%	H. Sun acknowledges the support of
%	NSF of China under grant No. \,11701563.	{The authors also would like to thank all the anonymous referees for their detailed comments that helped us to improve the manuscript.}
%}

%\begin{acknowledgements}
%\end{acknowledgements}




\end{document}